\documentclass[10pt]{amsart}

\usepackage{amsmath}
\usepackage{amssymb}

\newtheorem{proposition}{Proposition}[section]
\newtheorem{theorem}[proposition]{Theorem}
\newtheorem{corollary}[proposition]{Corollary}
\newtheorem{lemma}[proposition]{Lemma}

\theoremstyle{definition}
\newtheorem{definition}[proposition]{Definition}

\newtheorem{remark}[proposition]{Remark}
\newtheorem{example}[proposition]{Example}

\numberwithin{equation}{section}

\newcommand{\R}{\mathbb R}

\def\Dx{\Delta_x}
\def\Cal{\mathcal}
\def\({\left(}
\def\){\right)}
\def\Nx{\nabla_x}
\def\Dt{\partial_t}
\def\dist{\operatorname{dist}}

\def\eb{\varepsilon}
\def\dist{\operatorname{dist}}

\def\Bbb{\mathbb}
\def\spann{\operatorname{span}}
\def\meas{\operatorname{meas}}
\def\divv{\operatorname{div}}
\begin{document}

\title[Strong Uniform Attractors]{Strong Uniform Attractors for Non-Autonomous Dissipative PDEs with non translation-compact external forces}
\author[S.Zelik]{Sergey Zelik}

\begin{abstract}We give a comprehensive study of strong uniform attractors of non-autonomous dissipative systems for the case where the external forces are not translation compact. We introduce several new classes of external forces which are not translation compact, but nevertheless allow to verify the attraction in a strong topology of the phase space and discuss in a more detailed way the class of so-called normal external forces introduced before. We also develop a unified approach to verify the asymptotic compactness for such systems based on the energy method and apply it to a number of equations of mathematical physics including the Navier-Stokes equations, damped wave equations and reaction-diffusing equations in unbounded domains.
\end{abstract}

\subjclass[2000]{35B40, 35B45}
\keywords{non-autonomous equations, uniform attractors, asymptotic compactness  }
\address{University of Surrey, Department of Mathematics, \newline
Guildford, GU2 7XH, United Kingdom.}

\email{s.zelik@surrey.ac.uk}

\maketitle
\tableofcontents

\section{Introduction}\label{s1}

It is well-known that in many casees the long-time behaviour of dissipative partial differential equations (PDEs) can be described in terms of the so-called {\it global attractors}. Being a compact invariant subset of a phase space attracting the images of all bounded sets when time tends to infinity, a global attractor contains all the non-trivial dynamics of the system considered. On the other hand, it is usually essentially smaller that the initial phase space. In particular, in many cases this attractor has finite Hausdorff and fractal dimension, so despite the infinite-dimensionality of the initial phase space (e.g., $\Phi=L^2(\Omega)$), the limit reduced dynamics on the attractor is in a sense finite-dimensional and can be described by finitely many parameters, see \cite{BV,Tem, CVbook} and references therein.
\par
The situation becomes  more complicated when the considered dissipative PDE is {\it non-autonomous}, for instance contains the external forces depending explicitly on time), e.g., when the underlying PDE has the form
\begin{equation}\label{0.PDE}
u_t=A(u)+g(t),\ u\big|_{t=\tau}=u_\tau\in W,
\end{equation}
where $A(u)$ is some non-linear operator which we will not specify  here, see the examples in Section \ref{s4} below, $W$ is a phase space of the problem considered (we will assume below that $W$ is a reflexive Banach space) and the external forces $g$ are assumed to belong to the space
$L^p_b(\R,V)$ with $1<p<\infty$, where $V$ is another reflexive Banach space, see Section \ref{s1} for more detailed exposition.
\par
At present time there exist two major ways to generalize the concept of a global attractor to the case of non-autonomous PDEs. The first one treats
the attractor of a non-autonomous system as a time-dependent set as well $\Cal A=\Cal A(t)\subset W$, $t\in\R$. This naturally leads to the so-called
{\it pullback} attractors or the kernel sections in the terminology of Vishik and Chepyzhov, see \cite{CLR,CVbook,CF} and references therein. One of the main advantages of this approach is the fact that the attractor $\Cal A(t)$ usually remains finite-dimensional for every $t$ and it is also well adapted to study {\it random/stochastic} PDEs, see \cite{CF}. However, in the case of deterministic PDEs, this approach has an essential drawback, namely, the attraction forward in time is usually lost and we have only a weaker form of the attraction property backward in time. As a result, an exponentially repelling forward in time trajectory  may be a pullback attractor, see \cite{MZ} and references therein. Mention also that this problem can be overcome using the concept of the so-called non-autonomous {\it exponential} attractor, see \cite{EMZ}.
\par
The alternative approach which is based on the reduction of the non-autonomous dynamical system (DS) to the autonomous one treats the attractor of a non-autonomous DS as a time-independent set $\Cal A\subset W$. This approach naturally leads to the so-called {\it uniform} attractor which is the main object of investigation of the present paper, so we explain it in a bit more detailed way. Following the general scheme, one should consider the  family of equations of the form \eqref{0.PDE}:
\begin{equation}\label{0.PDE-hull}
u_t=A(u)+h(t),\ u\big|_{t=\tau},\ h\in\Cal H(g)
\end{equation}
with all external forces belonging to the so-called hull $\Cal H(g)$ of the initial external force $g$ generated by all time shifts of the initial external force $g$ and their closure in the properly chosen topology, see Sections \ref{s2} and \ref{s3} for more details.
\par
Then, assuming that the problems \eqref{0.PDE-hull} are globally well-posed in $W$, we have a family of dynamical processes $U_h(t,\tau):W\to W$, $h\in\Cal H(g)$, in the phase space $W$ generated by the solution operators $U_h(t,\tau)u_\tau:=u(t)$ of \eqref{0.PDE-hull}. Introduce an extended phase space
$\Phi:=W\times\Cal H(g)$ associated with problem \eqref{0.PDE-hull}. Then, the extended semigroup on $\Phi$ is defined as follows:
\begin{equation}
\Bbb S(t)(u_0,h):=(U_h(t,0)u_0,T(t)h),\ u_0\in W,\ h\in\Cal H(g),\ (T(s)h)(t):=h(t+s).
\end{equation}
Finally, if this semigroup possesses a global attractor $\Bbb A\subset \Phi$, then its projection $\Cal A:=\Pi_1\Bbb A$ is called a {\it uniform} attractor associated with the family of equations \eqref{0.PDE-hull}, see \cite{Chep13,CVbook,CV95,CV94} as well as Section \ref{s3} for more details.
\par
Clearly, the choice of the topology on the extended phase space is crucial for this approach. According to Vishik and Chepyzhov, see \cite{CVbook}, there are two natural choices of this topology. First one is the  weak topology which leads to the so-called {\it weak} uniform attractor. In this case, we take the topology induced by the embedding $\Cal H(g)\subset L^p_{loc,w}(\R,V)$ on the hull $\Cal H(g)$ (this gives the compactness of the hull since the space $L^p_{loc}(\R,V)$ is reflexive and bounded sets are precompact in a weak topology) and also the weak topology on the phase space $W$. In this case we need not extra assumptions on the external forces $g$ and only the translation boundedness: $g\in L^p_b(\R,V)$ is usually sufficient to have a weak uniform attractor~$\Cal A$.
\par
The second natural choice is the choice of {\it strong} topologies on both components $W$ and $\Cal H(g)$. In this case, we need the
 extra assumption that the hull $\Cal H(g)$ is compact in a {\it strong} topology of $L^p_{loc}(\R,V)$ (the external forces satisfying this condition are usually referred as translation-compact). Thus, this alternative choice requires the extra assumption for the external forces to be translation-compact and gives the so-called {\it strong} uniform attractor $\Cal A\subset W$, see \cite{Chep13,CVbook} for many applications of this theory for various equations of mathematical physics.
\par
However, as has been pointed out later, there is one more a bit surprising choice of the topologies when one takes the {\it strong} topology on the $W$-component of the phase space $\Phi$ and the {\it weak} topology on the hull $\Cal H(g)$. Then, it is possible in many cases to verify the existence of a {\it strong} uniform attractor $\Cal A\subset W$ for the case when the external forces $g$ are not translation compact. Of course, to gain this strong compactness, we need some extra assumptions on $g$, but these assumptions can be essentially weaker than the translation compactness. The most known example here is the so-called {\it normal} external forces which usually give the strong compactness in the case of parabolic PDEs in bounded domains, see \cite{AnQ,LWZ,Lu} and  Sections \ref{s2} and \ref{s4} for more details. Mention also the paper \cite{MZS} where the weaker than normal class of external forces (which is close to the necessary one to have the strong attractor) is introduced, see also Section \ref{s1}.
\par
Note that, the usage of normal external forces is restricted to parabolic equations and for more complicated problems, e.g., damped wave equations this assumption is not sufficient to give the existence of a strong attractor, see counterexamples of Section \ref{s5}. Moreover, to the best of our knowledge only one reasonable class of non-translation compact external forces which give the existence of strong uniform attractors for non-parabolic problems (but with the conditions which is far from being necessary for the asymptotic compactness) has been introduced in the literature, see \cite{MCL} and \cite{MZh}.
\par
The main aim of the present paper is  to present some new classes of non-translation compact external forces which guarantee the existence of a strong attractor in many cases of equations of mathematical physics not restricted to parabolic problems only. Our approach is based on the well-known fact that the translation compact external forces are those which can be approximated (in $L^p_b(\R,V)$) by smooth functions in {\it space and time} (at least in the case when $V$ is a properly chosen Sobolev space). Following this idea, we introduce the class of {\it time regular} and {\it space regular} external forces which consist of functions which can be approximated in $L^p_b(\R,V)$ by the functions smooth {\it only} in time or {\it only} in space respectively, see Section \ref{s1} for the rigorous definition. Note that our class of space regular functions is very close (may be even equivalent) to the class introduced in \cite{MCL}. In contrast to this, our class of time regular functions seems to be indeed new.
 \par
 Then, on the one hand, as expected, the intersection of these two classes of external forces coincides with the translation compact external forces. On the other hand, in many cases, including e.g., damped wave
equations only time regularity or only space regularity is sufficient to gain the asymptotic compactness and prove the existence of a strong uniform attractor, see Section \ref{s4} for the details. Moreover, in order to treat normal external forces in a similar way, we introduce a slightly more restrictive class of {\it strongly} normal external forces consisting of functions which can be approximated in $L^p_b(\R,V)$ by functions belonging to $L^\infty(\R,V)$.
\par
Finally, we have developed a unified approach to check the asymptotic compactness for all classes of non translation compact external forces mentioned above based on the so-called energy method, see \cite{ball} and \cite{MoRosa}. To the best of our knowledge, the energy method has been not used before in the case when the external forces are not translation compact and essentially more complicated methods are used instead. We illustrate this approach on a number of equations of mathematical physics including the Navier-Stokes equations, damped wave equations and reaction-diffusion equations in unbounded domains, see Section \ref{s4} for the details.
\par
The paper is organized as follows.
\par
The new classes of external forces are introduced in Section \ref{s1}. We also study the relation between them here.
\par
 The properties of the weak hull $\Cal H(g)$ are studied in Section \ref{s2} in the case where the function $g$ belongs to one of the classes introduced in Section \ref{s1}. In this section, we introduce the main technical tools which are necessary in order to check the asymptotic compactness for our non translation compact external forces.
 \par
 In Section \ref{s3}, we remind briefly the main concepts and main theorems of the theory of uniform attractors.
 \par
 Section \ref{s4} contains our applications of the general theory to concrete classes of equations of mathematical physics.
 \par
 Finally, in Section \ref{s5}, we give several examples of natural classes of external forces which are however  {\it not sufficient} to have the asymptotic compactness and the existence of a strong attractor.

\section{Classes of admissible external forces}
In this section, we introduce a number of classes of external forces $g$ which will be used throughout of the paper and study the relations between them. We assume that we are given a {\it reflexive} Banach space $V$ and a function (external force) $g:\R\to V$ such that
\begin{equation}\label{1.lpv}
g\in L^p_b(\R,V), \ 1<p<\infty,
\end{equation}
where the uniformly local space $L^p_b(\R,V)$ is defined by the following norm:
\begin{equation}\label{1.ul}
\|g\|_{L^p_b(\R,V)}:=\sup_{s\in\R}\|u\|_{L^p((s,s+1),V)}<\infty.
\end{equation}
Functions belonging to the space $L^p_b(\R,V)$ are often refereed as translation bounded, see e.g. \cite{CVbook}.
\par
Unfortunately, the sole translation boundedness is not sufficient to gain the existence of a {\it strong} uniform attractor for the corresponding dissipative PDE, see e.g. Example \ref{Ex5.par} below, so some stronger conditions on $g$ should be posed. We start our exposition by considering the so-called {\it translation-compact} functions introduced by Vishik and Chepyzhov, see \cite{CV95,CVbook}, which is, in a sense, the strongest assumption to be made which guarantees the existence of a strong attractor for various classes of dissipative PDEs.

\subsection{Translation compact external forces} Remind that the space $L^p_b(\R,V)$ is invariant with respect to time translations:
\begin{equation}\label{1.tr}
(T(s)g)(t):=g(t+s),\ \ t,s\in\R,
\end{equation}
therefore, we may speak about the orbit of $g\in L^p_b(\R,V)$ under the translation group:
 \begin{equation}\label{1.orbit}
 \mathcal O(g):=\{T(s)g,s\in\R\}\subset L^p_b(\R,V).
 \end{equation}
 A function $g\in L^p_b(\R,V)$ is {\it translation-compact} if its orbit $\mathcal O(g)$ is pre-compact in the Frechet space $L^p_{loc}(\R,V)$.
\par
The next proposition gives the criterion for the function $g$ to be translation compact.

\begin{proposition}\label{Prop1.trcomp} A function $g\in L^p_b(\R,V)$ is translation compact if and only if
\par
1) $g$ has an $L^p$ modulus of continuity, i.e. there exists a monotone  function $\omega:\R_+\to\R_+$ such that $\lim_{z\to0}\omega(z)=0$ and
\begin{equation}\label{1.mc}
\sup_{t\in\R}\int_t^{t+1}\|g(s+\tau)-g(s)\|^p_V\,ds\le \omega(\tau).
\end{equation}
2) Let $g_h(t):=\frac1h\int_t^{t+h}g(s)\,ds$. Then the range of $g_h$ is precompact in $V$ for any positive $h$, i.e. the set
\begin{equation}\label{1.range}
\bigg\{g_h(t),\,t\in\R\bigg\}\subset V
\end{equation}
is pre-compact in $V$ for every fixed $h>0$.
\end{proposition}
The proof of Proposition \ref{Prop1.trcomp} is based on the proper version of the Arzela-Ascoli theorem for $L^p$-spaces and is given, e.g. in \cite{CVbook}.
\par
Since we are mainly interested in this paper in the classes of external forces which are {\it not} translation compact, but nevertheless may lead to the existence of strong attractors, we finish our exposition of translation compact functions, see \cite{CVbook} for more details and more examples, and switch to more interesting from our point of view classes of functions.

\subsection{Space regular external forces} Roughly speaking, translation-compact external forces are those which are regular (can be approximated by smooth functions) in space and time. In this and next subsection, we introduce new classes of external forces which regular only in space or in time. As we will see below such partial regularity is sufficient for the asymptotic compactness for many interesting classes of dissipative PDEs.
\begin{definition} A function $g\in L^p_b(\R,V)$ is space regular if for every $\eb>0$ there exists a finite-dimensional subspace $V_\eb\subset V$, $\dim V_\eb<\infty$ and a function $g_\eb\in L^p_b(\R,V_\eb)$ such that
\begin{equation}\label{1.spr}
\|g-g_\eb\|_{L^p_b(\R,V)}\le \eb.
\end{equation}
In other words, a function $g\in L^p_b(\R,V)$ is space regular if it can be approximated by the functions with finite-dimensional range.
\end{definition}
The next proposition gives a typical example for space regular functions.
\begin{proposition}\label{Prop1.s-reg} Let the space $W\subset\subset V$ is compactly embedded in $V$ and let $g\in L^p_b(\R,W)$. Then the function $g$ is space regular as a function in $L^p_b(\R,V)$.
\end{proposition}
\begin{proof} Since $W$ is compactly embedded in $V$, the unit ball $B(0,1,W)$ of $W$ is precompact in $V$. Therefore, by the Hausdorff criterion, for every $\eb>0$, we may construct an $\eb$-net $\{x_1,\cdots, x_n\}$. Define
\begin{equation*}
V_\eb:=\spann\{x_1,\cdots,x_n\}.
\end{equation*}
We also may define a "projector" $P: B(0,1,W)\to V_\eb$ such that
\begin{equation*}
\|x-P(x)\|_V\le\eb,\ \ x\in B(0,1,W).
\end{equation*}
In a fact, this "projector" even can chosen to be linear, but this is not important for us. Extending this projector by scaling to any $R$-ball $B(0,R,W)$, we finally get
\begin{equation}\label{1.pr}
\|x-P(x)\|_V\le\eb\|x\|_{W},\ \ \forall x\in W.
\end{equation}
We now may define $g_\eb(t):=P(g(t))$. Then, $g_\eb\in L^p_b(\R,V_\eb)$ and
\begin{equation*}
\|g-g_\eb\|_{L^p_b(\R,V)}\le \eb\|g\|_{L^p_b(\R,W)}\le C\eb
\end{equation*}
and the proposition is proved.
\end{proof}
\begin{remark}\label{Rem1.smooth} In applications usually $V=L^p(\Omega)$ or some Sobolev space $V:=W^{s,p}(\Omega)$. In this case, by arbitrarily small perturbation of the finite-dimensional spaces $V_\eb$ involved in the definition of space regularity can be made smooth: $V_\eb\subset C^\infty(\Omega)$. Thus, without loss of generality, we may assume that the approximating functions $g_\eb$ are smooth in space. This simple observation is however crucial for proving asymptotic compactness for the case of space regular external forces.
\end{remark}
\subsection{Time regular external forces} In this subsection, we introduce the complementary class of functions which are regular in time, but in general, not in space.
\begin{definition}\label{Def1.t-reg} A function $g\in L^p_b(\R,V)$ is time regular if for any $\eb>0$ there exists a function $g_\eb\in C^k_b(\R,V)$ for all $k>0$, such that
\begin{equation*}
\|g-g_\eb\|_{L^p_b(\R,V)}\le\eb.
\end{equation*}
In other words, a function $g\in L^p_b(\R,V)$ is time regular if it can be approximated in $L^p_b(\R,V)$ by functions smooth in time.
\end{definition}
\begin{remark}\label{Rem1.hc} The assumption $g_\eb\in C^k_b(\R,V)$ is sometimes not convenient since the spaces $C^k$ are not reflexive and we do not have weak compactness of a unit ball. By this reason, instead of $C^k$, we will use the assumption $g_\eb\in H^k_b(\R,V)$, where $H^k$ is a Sobolev space of functions whose derivatives up to order $k$ belong to $L^2$. On the one hand, it does make essential difference due to the Sobolev
embedding $H^{k+1}\subset C^k$ (1D case!). On the other hand, the space $H^k$ is reflexive and we may use weak compactness of a unit ball, e.g., in order to prove that this norm preserve under various weak limits, see next section for more details.
\end{remark}
The next proposition gives an important characterization of time regular external forces.
\begin{proposition}\label{Prop1.mc} A function $g\in L^p_b(\R,V)$ is time regular if and only if it possesses a modulus of continuity in $L^p_b$, i.e. there exists a monotone decreasing function $\omega:\R_+\to\R_+$ such that $\lim_{z\to0}\omega(z)=0$ and
\begin{equation}\label{1.spr-mc}
\sup_{t\in\R}\int_{t}^{t+1}\|g(s+\tau)-g(s)\|_V^p\,ds\le\omega(\tau).
\end{equation}
\end{proposition}
\begin{proof} Indeed, let $g\in L^p_b(\R,V)$ be time regular. Then, it can be approximated in $L^p_b(\R,V)$ by functions smooth in time. Every such smooth in time function has modulus of continuity and the elementary arguments show that it preserves under the limit.
\par
Vise versa, let the function $g\in L^p_b(\R,V)$ has a modulus of continuity. Then, we just construct the smooth approximations $g_\eb(t)$ to $g$ using the standard mollification operator in time:
\begin{equation*}
g_\eb(t):=\int_{R}\varphi_\eb(s)g(t-s)\,ds,
\end{equation*}
where the  mollification kernels satisfy $\phi_\eb(z):=\frac1\eb\varphi(z\eb^{-1})$, $\varphi\in C_0^\infty(0,1)$ and $\int_{R}\varphi(z)\,dz=1$. Then, $g_\eb(t)$ are smooth in time and as easy to show using the existence of modulus of continuity, that
\begin{equation*}
\|g-g_\eb\|_{L^p_b(\R,V)}\to0,\ \ \eb\to0.
\end{equation*}
This proves the proposition.
\end{proof}
The next theorem shows the relations between the introduced space regular and time regular functions and translation compactness.
\begin{theorem}\label{Th1.main} A function $g\in L^p_b(\R,V)$ is space regular and time regular if and only if it is translation compact.
\end{theorem}
\begin{proof} The fact that translation compactness implies time regularity follows from Propositions \ref{Prop1.trcomp} and \ref{Prop1.mc}. Let us check that the translation compact function is space regular. Indeed, as not difficult to check, the translation compactness implies the convergence
\begin{equation*}
\|g-g_h\|_{L^p_b(\R,V)}\to 0,\ \ h\to0,
\end{equation*}
where the approximations $g_h$ are defined in Proposition \ref{Prop1.trcomp}. Moreover, due to the second part of Proposition \ref{Prop1.trcomp},
$f_h\in C_b(\R,K_h)$ for some {\it compact} set $K_h$ of $V$. Since a compact set is almost finite-dimensional,
arguing similarly to Proposition \ref{Prop1.s-reg}, we may find a finite dimensional space  $V_h$ such that $\dist(K,V_h)\le\eb$. Projecting the function $g_h$ to this finite-dimensional plane, we split it as follows: $g_h=g^1_h+r_h$, where $g^1_h\in C_b(\R,V_h)$ and
\begin{equation*}
\|r_h\|_{L^p_b(\R,V)}\le\eb.
\end{equation*}
Thus, the desired approximations of $g$ by functions with finite-dimensional range are constructed and $g$ is space  regular.
\par
Let us assume now that $g\in L^p_b(\R,V)$ is simultaneously space and time regular. We need to prove that $g$ is translation compact.
To do this, we will check the assumptions of the criterion of Proposition \ref{Prop1.trcomp}. Indeed, the first assumption is automatically satisfied since $g$ is time regular, see Proposition \ref{Prop1.mc}. Let us verify the second one. Indeed, since $g$ is space regular, we may approximate it by functions $g^n(t)$ of finite range $g^n\in L^p_b(\R,V_n)$, $\dim V_n<\infty$ and
\begin{equation*}
\|g-g^n\|_{L^p_b(\R,V)}\to 0,\ \ n\to\infty.
\end{equation*}
Introduce the functions $g^n_h(t):=\frac1h\int_t^{t+h}g^n(s)\,ds$. Then, obviously, for every fixed $h>0$, we have the convergence
\begin{equation*}
g^n_h\to g_h\ \text{ in }\ C_b(\R,V),\ \ n\to\infty.
\end{equation*}
Moreover, since $V_n$ is finite-dimensional, the set $\{g^n_h(t),\,t\in\R\}\subset V$ is precompact in $V$ for any $n$. Finally,
the set $\{g_h(t),\,t\in\R\}\subset V$ is precompact in $V$ as a limit of precompact sets. This gives the translation compactness of $g$ and finishes the proof of the theorem.
\end{proof}
\begin{remark}\label{Rem1.sp-time-compact} As we can see, two  introduced classes of external forces (space regular and time regular) intersect {\it exactly} by the class of translation compact functions.
\end{remark}

\subsection{Normal, weakly normal and strongly normal external forces} Normal external forces was the first big class of non translation-compact external forces introduced in \cite{LWZ} which gives the strong asymptotic compactness for some dissipative PDEs. Unfortunately, its application is restricted to parabolic equations only and it is useless e.g., for damped hyperbolic equations.
\begin{definition} A function $g\in L^p_b(\R,V)$ is normal if
\begin{equation}\label{1.norm-def}
\sup_{t\in\R}\int_t^{t+\tau}\|g(s)\|^p_V\, ds\to0
\end{equation}
as $\tau\to0$. A typical example of a normal function is $g\in L^{p+\eb}_b(\R,V)$ for some $\eb>0$, although more complicated examples are also available.
\end{definition}
The applications of normal external forces to parabolic PDEs are based on the following simple lemma.
\begin{lemma}\label{Lem1.normalexp} Let $g\in L^p_b(\R,V)$ be normal. Then
\begin{equation}\label{1.exp}
\sup_{t\in\R} \int_t^{t+1}e^{-N(t-s)}g(s)\,ds\to0
\end{equation}
as $N\to\infty$.
\end{lemma}
The proof of the lemma is elementary and we leave it for the reader, see also~\cite{LWZ}.
\par
We now introduce one more class of external forces which are weaker than the normal ones and which somehow unify the classes of space regular, time regular and normal external forces. This class is very close (maybe even equivalent) to the one introduced in \cite{MZS}.
\begin{definition}\label{Def1.w-normal} Let $g\in L^p_b(\R,V)$ be a translation bounded external force. Then, the external force $g$ is weakly normal if for any $\eb>0$ there exist $\tau=\tau(\eb)>0$, a {\it finite-dimensional} space $V_\eb\subset V$ and a function $g_\eb\in L^p_b(\R,V_\eb)$ such that
\begin{equation}\label{1.weak-norm-def}
\sup_{t\in\R}\int_t^{t+\tau}\|g(s)-g_\eb(s)\|_{V}^p\,ds\le \eb.
\end{equation}
\end{definition}
The analogue of Lemma \ref{Lem1.normalexp} then reads.
\begin{lemma}\label{Lem1.w-normalexp} Let $g\in L^p_b(\R,V)$ be translation bounded and weakly normal. Then, there exists a monotone  function $\omega:\R_+\to\R_+$ such that $\lim_{z\to0}\omega(z)=0$ such that
\begin{equation}\label{1.wn-exp}
\limsup_{N\to\infty}\sup_{t\in\R}\int_{t-1}^te^{-N(t-s)}\|g(s)-g_\eb(s)\|^p_{V}\,ds\le \omega(\eb).
\end{equation}
\end{lemma}
The proof of this lemma is straightforward and we left it to the reader, see  \cite{MZS}.
\par
The next proposition shows that all of the classes of external forces introduced before are contained in the class of weakly normal external forces.
\begin{proposition}\label{Prop1.evident} Let the external force $g\in L^2_b(\R,V)$ be time regular or space regular or normal. Then $g$ is weakly normal.
\end{proposition}
All assertions of this propositions are immediate corollaries of the corresponding definitions, so we again left the rigorous proof to the reader.
 \par
   As we have seen, the definition of normal external forces has a different structure than the definition of space or time regular external forces given before and actually we do not know how to define normal external forces using approximations. Instead, we introduced slightly stronger class of external forces which contain normal ones and which can be described on the level of approximations.
\begin{definition}\label{Def1.snorm} A function $g\in L^p_b(\R,V)$ is strongly normal if for any $\eb>0$ there exists a function $g_\eb\in L^\infty(\R,V)$ such that
\begin{equation*}
\|g-g_\eb\|_{L^p_b(\R,V)}\le\eb.
\end{equation*}
\end{definition}
The next proposition clarifies the relations between the introduced concepts.
\begin{proposition}\label{Prop1.snorm}\par 1) Let $g\in L^p_b(\R,V)$ is strongly normal. Then it is normal.
\par
2) Let $g\in L^{p+\eb}_b(\R,V)$ for some $\eb>0$. Then it is strongly normal in $L^p_b(\R,V)$.
\par
3) Let $g\in L^p_b(\R,V)$ is time regular. Then it is strongly normal.
\end{proposition}
\begin{proof} Let $g$ be strongly normal in $L^p_b(\R,V)$. Then it can be approximated in this space by functions $g_n\in L^\infty(\R,V)$. Since every $g_n$ is obviously normal in $L^p_b(\R,V)$ and this property preserves under passing to the limit (as not difficult to see by elementary reasons) the limit function $g$ is also normal.
\par
Let now $g\in L^{p+\eb}(\R,V)$ for some $\eb>0$. We need to approximate it by function from $L^\infty(\R,V)$. To this end, we introduce a sequence of functions
\begin{equation}
g_N(t):=\begin{cases} g(t),\ \ \|g(t)\|_V\le N,\\ 0,\ \ \|g(t)\|_V>N.\end{cases}
\end{equation}
Then, clearly $g_N\in L^\infty(\R,V)$. Let us now estimate the difference $g-g_N$ using the H\"older and Chebyshev inequalities:
\begin{multline}
\int_t^{t+1}\|g-g_N\|^p_V\,ds=\int_{(t,t+1)\cap\{\|g\|_V>N\}}\|g(s)\|^p_V\,ds\le\\\le \|g\|_{L^{p+\eb}_b(\R,V)}^\theta\(\meas\{(t,t+1)\cap\{\|g(t)\|_V>N\}\}\)^{1-\theta}\le C\|g\|_{L^1(t,t+1,V)}^{1-\theta}N^{\theta-1}.
\end{multline}
for some $0<\theta<1$ depending only on $p$ and $\eb$. Thus, indeed, the difference $g-g_N$ tends to zero in $L^p_b(\R,V)$ and $g$ is strongly normal in $L^p_b(\R,V)$.
\par
The third assertion is obvious since a time regular function can be approximated by functions smooth in time, in particular, these approximations belong to $L^\infty(\R,V)$.
\end{proof}

\section{Weak hulls for translation bounded external forces}\label{s2}
In this section, we prepare a number of technical tools which will allow us to check asymptotic compactness in the case where the external forces belong to the classes introduced in the previous section. We start with introducing weak hull of a translation bounded function $g\in L^p_b(\R,V)$.
In contrast to the translation compact case, now the orbit $\mathcal O(g)$, see \eqref{1.orbit}, is not precompact in a strong topology of $L^p_{loc}(\R,V)$. However, since the space $L^p_{loc}(\R,V)$ is {\it reflexive}, see e.g. \cite{RR} (we remind that thought of the paper the space $V$ is reflexive and $1<p<\infty$), the orbit $\mathcal O(g)$ is precompact in a {\it weak} topology of $L^p_{loc}(\R,V)$. This space endowed by the weak topology, we will denote by $L^p_{loc,w}(\R,V)$. Remind that a sequence $v_n\to v$ in $L^p_{loc,w}(\R,V)$ if for every time interval $[-N,N]$, the sequence $v_n\big|_{t\in[-N,N]}$ is convergent weakly in $L^p((-N,N),V)$ to the function $v\big|_{t\in[-N,N]}$. This weak compactness justifies the following definition.
\begin{definition}\label{Def2.hull} Let $g\in L^p_b(\R,V)$ be translation bounded. Then the weak hull $\mathcal H(g)$ is defined as a closure of the orbit $\mathcal O(g)$ in the space $L^p_{loc,w}(\R,V)$:
\begin{equation}\label{2.hull}
\mathcal H(g):=\bigg[\mathcal O(g)\bigg]_{L^p_{loc,w}(\R,V)}.
\end{equation}
Clearly, the group of translations acts of the hull:
\begin{equation}\label{2.hulltr}
T(s):\mathcal H(g)\to\mathcal H(g),\ s\in\R,\ \  T(s)\mathcal H(g)=\mathcal H(g).
\end{equation}
This simple property is important for the reduction of a non-autonomous process to a semigroup, see the next section.
\end{definition}
The next proposition shows that the weak hull does not destroy the space or time regularity as well as normality introduced in the previous section.
\begin{proposition}\label{Prop2.hull} Let $g\in L^p_b(\R,V)$ be translation bounded. Then \par
1) All functions $h\in\mathcal H(g)$ are translation bounded and
\begin{equation}
\|h\|_{L^p_b(\R,V)}\le \|g\|_{L^p_b(\R,V)}
\end{equation}
for any $h\in\mathcal H(g)$.
\par
2) Let $g$ be time regular. Then all $h\in\mathcal H(g)$ are time regular. Moreover, the modulus of continuity of the external force $g$ is simultaneously a modulus of continuity for all external forces $h\in\mathcal H(g)$.
\par
3) Let $g$ be space regular. Then all $h\in\mathcal H(g)$ are space regular as well.
\par
4) Let $g$ be normal. Then all $h\in\mathcal H(g)$ are normal as well. Moreover,
\begin{equation}\label{2.hull-norm}
\sup_{h\in\mathcal H(g)}\sup_{t\in\R}\int_t^{t+\tau}\|h(s)\|^p\,ds\to 0,\ \ \text{ as }\ \ \tau\to0.
\end{equation}
\par
5) Let $g$ be strongly normal. Then all $h\in\mathcal H(g)$ are strongly normal as well.
\par
6) Let $g$ be weakly normal. Then, for every $h\in\Cal H(g)$ and every $\eb>0$ there exists $h_\eb\in \Cal H(g_\eb)\subset L^p_b(\R,V_\eb)$ such that
inequalities \eqref{1.weak-norm-def} and \eqref{1.wn-exp} hold uniformly with respect to $h\in\Cal H(g)$.
\end{proposition}
The proof of the proposition is straightforward but a bit technical, so, in order to avoid the technicalities, we leave it to the reader.
\par
The next simple proposition is however crucial for what follows.
\begin{proposition}\label{Prop2.app} Let $g\in L^p_b(\R,V)$ be translation bounded and let $g_\eb\in L^p_b(\R,V)$ be such that
\begin{equation}
\|g-g_\eb\|_{L^p_b(\R,V)}\le \eb.
\end{equation}
Then, for any $h\in\mathcal H(g)$, there exists $h_\eb\in\mathcal H(g_\eb)$ such that
\begin{equation}\label{2.heb}
\|h-h_\eb\|_{L^p_b(\R,V)}\le \eb.
\end{equation}
\end{proposition}
\begin{proof} Indeed, let $h\in \mathcal H(g)$. By definition, this means that there exists a sequence $s_n\in\R$ such that
\begin{equation*}
h=\lim_{n\to\infty}T(s_n)g,
\end{equation*}
where the limit is understood as a limit in $L^p_{loc,w}(\R,V)$. Since the hull $\mathcal H(g_\eb)$ is a compact set in $L^p_{loc,w}(\R,V)$, up to passing to a subsequence, we may assume that the sequence $T(s_n)g_\eb$ is convergent. Finally, we set
\begin{equation*}
h_\eb:=\lim_{n\to\infty}T(s_n)g_\eb.
\end{equation*}
It is easy to see that the constructed $h_\eb$ satisfies \eqref{2.heb} and the proposition is proved.
\end{proof}
\begin{corollary}\label{Cor2.time} Let $g\in L^p_b(\R,V)$ be time regular and let $h_n\in \mathcal H(g)$ be such that $h_n\to h$ in $L^p_{loc,w}(\R,V)$. Then, up to passing to the subsequence, for every $\eb>0$, there exists a sequence $\varphi_n\in H^k_b(\R,V)$, for any $k\in\Bbb N$, such that $\varphi_n\to\varphi$ in $H^k_{loc,w}(\R,V)$ and
\begin{equation}
\|h_n-\varphi_n\|_{L^p_b(\R,V)}\le\eb,\ \ \|h-\varphi\|_{L^p_b(\R,V)}\le\eb.
\end{equation}
\end{corollary}
\begin{proof} Indeed, by definition of time regularity, we may approximate $g\in L^p_b(\R,V)$ by $g_\eb\in H^k_b(\R,V)$, $k\in\Bbb N$ with any accuracy $\eb$. Due to Proposition \ref{Prop2.app}, we may approximate any $h_n\in\mathcal H(g)$ by the functions $\varphi_n\in\mathcal H(g_\eb)$ again with any accuracy $\eb>0$. It remains to note that $\mathcal H(g_\eb)\subset H^k_b(\R,V)$ and is compact in $H^k_{loc,w}(\R,V)$ for any $k\in\Bbb N$. Thus, passing to a subsequence if necessary, we have the convergence $\varphi_n\to \varphi$ in $H^k_{loc,w}(\R,V)$ for any $k\in\Bbb N$ and the corollary is proved.
\end{proof}
The next corollary establishes the analogous result for space regular functions.

\begin{corollary}\label{Cor2.space} Let $g\in L^p_b(\R,V)$ be space regular, $V:=H^s(\Omega)$, $s\in\R$, be a Sobolev space of distributions over the domain $\Omega\subset\R^n$ and let $h_n\in\mathcal H(g)$ be such that $h_n\to h$ in $L^p_{loc,w}(\R,V)$. Then, up to passing to a subsequence, for every $\eb>0$, there exists a sequence $\varphi_n\in L^p_b(\R,H^k(\Omega))$, for all $k\in\Bbb N$ such that $\varphi_n\to\varphi$ in the space $L^p_{loc,w}(\R,H^k(\Omega))$ and
\begin{equation}
\|h_n-\varphi_n\|_{L^p_b(\R,V)}\le\eb,\ \ \|h-\varphi\|_{L^p_b(\R,V)}\le\eb.
\end{equation}
\end{corollary}
\begin{proof} Indeed, analogously to the previous corollary, we have a sequence $\varphi_n\in\mathcal H(g_\eb)$, with $g_\eb\in L^p_b(\R,V_n)$ where
$V_n$ is a {\it finite-dimensional} subspace of $V$. Moreover, without loss of generality, we may assume that $V_n\subset H^k(\Omega)$ for any $k\in\Bbb N$. Thus, using the fact that all norms are equivalent on a finite-dimensional space, we have $\mathcal H(g_\eb)\subset L^p_b(\R,H^k(\Omega))$ and is compact in $L^p_{loc,w}(\R,H^k(\Omega))$. Then, passing to a subsequence if necessary, we get the desired weak convergence $\varphi_n\to\varphi$ and finish the proof of the corollary.
\end{proof}
\section{Dynamical processes and uniform attractors: a brief reminder}\label{s3}
In this section, we briefly recall the main concepts related with non-autonomous dynamical systems and related uniform attractors, see e.g. \cite{CVbook} for more detailed exposition. We start with reminding the concept of a dynamical process related with the non-autonomous system.
\begin{definition}\label{Def3.pro} Let $W$ be a reflexive Banach space. A two parametric family of operators $U(t,\tau):W\to W$, $\tau\in\R$, $t\ge\tau$, is a dynamical process in $W$ if the following identities hold:
\begin{equation}\label{3.proc}
U(\tau,\tau)=\operatorname{Id},\ \ U(t,\tau)=U(t,s)\circ U(s,\tau),\ \ t\ge s\ge\tau,\ \tau\in\R.
\end{equation}
\end{definition}
Typical situation where a dynamical process naturally appear is the following one. Assume that we are given a non-linear (dissipative) PDE which we write in the following form
\begin{equation}\label{3.PDE}
\partial_t u=A(u)+g(t),\ \ u\big|_{t=\tau}=u_\tau,
\end{equation}
where $A(u)$ is a non-linear (unbounded) operator which we will not specify here (see examples in the next session) and $g(t)$ are the non-autonomous external forces which we assume to satisfy
\begin{equation}\label{3.force}
g\in L^p_b(\R,V)
\end{equation}
for some reflexive Banach space $V$ and some $1<p<\infty$.
\par
Assume also that problem \eqref{3.PDE} is globally well-posed in $W$, i.e., for every $u_\tau\in W$ there is a unique solution $u(t)\in W$, $t\ge\tau$. Then, as easy to see, the solution operators
\begin{equation}
U(t,\tau):W\to W,\ \  U(t,\tau)u_\tau:=u(t),\ \ t\ge\tau
\end{equation}
generate a dynamical process in $W$.
\par
To study the long time behavior of solutions of \eqref{3.PDE}, it is convenient (following to \cite{CVbook,CV95}) to consider not only \eqref{3.PDE} with a fixed right-hand side $g$, but also with all right-hand sides belonging to the hull $\mathcal H(g)$ of the initial external force:
\begin{equation}\label{3.PDEhull}
\partial_t u=A(u)+h(t),\ \ h\in\mathcal H(g),\ \ u\big|_{t=\tau}=u_\tau.
\end{equation}
Then, we have a family of dynamical processes $U_h(t,\tau):W\to W$ parametrized by $h\in\mathcal H(g)$ and this family satisfies the so-called translation identity:
\begin{equation}\label{3.trans}
U_h(t+s,\tau+s)=U_{T(s)h}(t,\tau),\  \ t\ge\tau,\ \ \tau,s\in\R.
\end{equation}
This identity allows us to reduce the non-autonomous dynamics to the autonomous one, but acting in the larger space. Namely, let
\begin{equation}\label{3.extended}
\Phi:=W\times \mathcal H(g)
\end{equation}
be the extended phase space associated with problem \eqref{3.PDEhull}. Then the extended semigroup on $\Phi$ is defined as follows:
\begin{multline}\label{3.sem}
\Bbb S(t)(u_0,h):=(U_h(t,0)u_0,T(t)h),\ (u_0,h)\in\Phi,\ t\ge0,\\ \Bbb S(t):\Phi\to\Phi,\ \ \Bbb S(t_1)\circ \Bbb S(t_2)=\Bbb S(t_1+t_2),\ t_i\ge0.
\end{multline}
Indeed, the semigroup property of \eqref{3.sem} is an immediate corollary of the translation identity \eqref{3.trans}. Thus, we may speak about global attractors of the semigroup $\Bbb S(t)$.

\begin{definition}\label{Def3.wattr} A set $\Bbb A\subset\Phi$ is a weak global attractor of $\Bbb S(t)$ if the following assumptions are satisfied:
\par
1) $\Bbb A$ is a compact set in $\Phi$ endowed by the weak topology endowed by the embedding $\Phi\subset W_w\times L^p_{loc,w}(\R,V)$. In the sequel, we will refer to this topology as~$\Phi_w$;
\par
2) $\Bbb A$ is strictly invariant: $\Bbb S(t)\Bbb A=\Bbb A$;
\par
3) $\Bbb A$ attracts the images of bounded sets of $\Phi$ in a weak topology $\Phi_w$, i.e., for every bounded set $\Bbb B$ in $\Phi$ and every neighborhood $\Bbb O(\Bbb A)$ of the set $\Bbb A$ in $\Phi_w$, there exists time $T=T(\Bbb B,\Bbb O)$ such that
\begin{equation}
\Bbb S(t)\Bbb B\subset\Bbb O(\Bbb A),\ \ t\ge T.
\end{equation}
Then, the projection of $\Bbb A$ to the first component ($W$) is called weak {\it uniform} attractor associated with the family of processes $U_h(t,\tau):W\to W$,\ $t\ge\tau$,\ $h\in\Cal H(g)$:
\begin{equation}
\Cal A:=\Pi_1\Bbb A.
\end{equation}
\end{definition}
To verify the existence of a weak attractor, we need some natural assumptions on the processes $U_h(t,\tau)$.

\begin{definition}\label{Def3.dis} A family of dynamical processes $U_h(t,\tau):W\to W$, $h\in\Cal H(g)$ is (uniformly) dissipative if the following estimate holds:
\begin{equation}\label{3.dis}
\|U_h(t,\tau)u_\tau\|_W\le Q(\|u_\tau\|_W)e^{-\beta(t-\tau)}+C_*,\  \tau\in\R,\ t\ge\tau,\ \ h\in\Cal H(g),
\end{equation}
where the positive constant $\beta$ and monotone function $Q$ are independent of $u_\tau\in W$, $h\in\Cal H(g)$ and $t\ge\tau$.
\par
A family $U_h(t,\tau):W\to W$ is weakly continuous if for any fixed $t$ and $\tau$, the weak convergence $u_\tau^n\to u_\tau$ in $W$ and $h_n\to h$ in $L^p_{loc,w}(\R,V)$ implies that
\begin{equation}\label{3.weakc}
U_{h_n}(t,\tau)u_\tau^n\to U_h(t,\tau)u_\tau
\end{equation}
weakly in $W$.
\end{definition}

\begin{theorem}\label{Th3.wa} Let the family of processes $U_h(t,\tau):W\to W$, $h\in\Cal H(g)$, associated with problems \eqref{3.PDEhull} be uniformly dissipative and weakly continuous. Then, the associated extended semigroup $\Bbb S(t):\Phi\to\Phi$ possesses a weak global attractor $\Bbb A$ and the associated uniform attractor $\Cal A=\Pi_1\Bbb A$ is generated by all bounded solutions of \eqref{3.PDEhull} with $h\in\Cal H(g)$:
\begin{equation}\label{3.str}
\Cal A=\cup_{h\in\Cal H(g)}\Cal K_h\big|_{t=0},
\end{equation}
where $\Cal K_h\subset L^\infty(\R,W)$ is a set of all trajectories $u(t)$, $t\in\R$, of problem \eqref{3.PDEhull} which are defined for all $t\in\R$ and remain bounded (the so-called kernel of equation \eqref{3.PDEhull} in the terminology of Vishik and Chepyzhov, see \cite{CVbook}).
\end{theorem}
For the rigorous proof of this theorem, see \cite{CVbook}. We just briefly mention here that it can be deduced from the standard attractor existence theorems for semigroups. Indeed, the dissipativity guarantees the existence of a bounded absorbing set for $\Bbb S(t)$. Then, the reflexivity of $W$ and $L^p_{loc}(\R,V)$ gives that this absorbing set is compact in $\Phi_w$ and the the weak continuity assumption implies that the operators $\Bbb S(t)$ are continuous in $\Phi_w$ for every fixed $t$. Thus, by the abstract attractor existence theorem, we have the existence of a weak global attractor $\Bbb A$ and representation \eqref{3.str} also follows from the fact that $\Bbb A$ is generated by all complete bounded trajectories of the associated semigroup $\Bbb S(t)$.

\begin{remark}\label{Rem3.uniform} There is an intrinsic definition of a uniform attractor $\Cal A$ which does not refer to the reduction of the dynamical process to a semigroup and does not uses hulls. Namely, a set $\Cal A\subset W$ is a  weak uniform attractor for the dynamical process $U_g(t,\tau):W\to W$ associated with equation \eqref{3.PDE} if
\par
1) $\Cal A$ is a compact set in a weak topology of $W$;
\par
2) For every bounded set $B\subset W$ and any neighborhood $\Cal O(\Cal A)$ of $\Cal A$ in a weak topology of $W$, there exists $T=T(B,\Cal O)$ such that
\begin{equation}
U_g(t,\tau)B\subset \Cal O(\Cal A),\ \text{ if }\ \ t-\tau>T,
\end{equation}
uniformly with respect to $\tau\in\R$.
\par
3) $\Cal A$ is a {\it minimal} set which satisfies 1) and 2).
\par
The equivalence of this definition and the one given before under the assumption of weak continuity of the processes $U_h(t,\tau)$ is proved in \cite{CVbook}. However, from our point of view the definition based on the reduction to a semigroup is more transparent, so we use it as the main one in this section.
\par
Mention also the alternative definition given in this remark allows to construct a uniform attractor even without the weak continuity. However, then
the key representation formula \eqref{3.str} may be lost and the theory becomes essentially less elegant.
\end{remark}
In applications of the next session, we always be in a situation where the existence of a weak uniform attractor is evident and verified in the previous papers, so we will mainly interested in under what extra assumptions on the external force $g$, we will have also attraction in a {\it strong} topology of $W$.

\begin{definition}\label{Def3.sattr} A set $\Bbb A\subset\Phi$ is a strong global attractor of $\Bbb S(t)$ if the following assumptions are satisfied:
\par
1) $\Bbb A$ is a compact set in $\Phi$ endowed by the  topology endowed by the embedding $\Phi\subset W\times L^p_{loc,w}(\R,V)$. In the sequel, we will refer to this topology as $\Phi_s$. Note that since $g$  is not assumed translation compact, so the topology on the hull $\Cal H(g)$ remains weak, but the strong topology on $W$ is chosen;
\par
2) $\Bbb A$ is strictly invariant: $\Bbb S(t)\Bbb A=\Bbb A$;
\par
3) $\Bbb A$ attracts the images of bounded sets of $\Phi$ in a strong topology $\Phi_s$, i.e., for every bounded set $\Bbb B$ in $\Phi$ and every neighborhood $\Bbb O(\Bbb A)$ of the set $\Bbb A$ in $\Phi_s$, there exists time $T=T(\Bbb B,\Bbb O)$ such that
\begin{equation}
\Bbb S(t)\Bbb B\subset\Bbb O(\Bbb A),\ \ t\ge T.
\end{equation}
Then, the projection of $\Bbb A$ to the first component ($W$) is called strong {\it uniform} attractor associated with the family of processes $U_h(t,\tau):W\to W$,\ $t\ge\tau$,\ $h\in\Cal H(g)$:
\begin{equation}
\Cal A:=\Pi_1\Bbb A.
\end{equation}
\end{definition}
To verify the existence of a strong uniform attractor, we need to pose some extra assumptions of the family of processes $U_h(t,\tau)$.
\begin{definition}\label{Def3.as} A family of dynamical processes $U_h(t,\tau):W\to W$, $h\in\Cal H(g)$, is (uniformly) asymptotically compact if for any sequence $u_n\in W$ bounded in $W$, any sequence of external forces $h_n\in\Cal H(g)$ and any $t_n\ge\tau_n$ such that $t_n-\tau_n\to\infty$, the sequence
\begin{equation}
\big\{U_{h_n}(t_n,\tau_n)u_n\big\}_{n=1}^\infty\subset W
\end{equation}
is precompact in $W$.
\end{definition}
\begin{theorem}\label{Th3.strong} Let the assumptions of Theorem \ref{Th3.wa} hold and let, in addition, the family of dynamical processes $U_h(t,\tau):W\to W$, $h\in\Cal H(g)$, be uniformly asymptotically compact. Then, the weak uniform attractor $\Cal A$ constructed in Theorem \ref{Th3.wa} is simultaneously a strong uniform attractor in $W$ (in the sense of Definition \ref{Def3.sattr}).
\end{theorem}
For the proof of this theorem, see \cite{CVbook,LWZ}.

\begin{remark}\label{Rem3.as-comp} As we can see, to verify the existence of a strong attractor, one basically needs to check the asymptotic compactness, so we will concentrate below on verification of  this property. Note also that the asymptotic compactness condition can be slightly simplified using the translation invariance of the family of processes $U_h(t,\tau): W\to W$, $h\in\Cal H(g)$. Namely, it is sufficient to check that for any sequence $u_n\in W$ bounded in $W$, any sequence of external forces $h_n\in\Cal H(g)$ and any sequence of initial times $\tau_n\to-\infty$, the sequence
\begin{equation}\label{3.ac-simple}
\big\{U_{h_n}(0,\tau_n)u_n\big\}_{n=1}^\infty\subset W
\end{equation}
is precompact in $W$. In the sequel, we will check the asymptotic compactness exactly in this form.
\end{remark}

\section{Applications: asymptotic compactness via the energy method}\label{s4}
In this section, applying the theory developed above, we verify the existence of a strong uniform attractors for a number of equations of mathematical physics. We will suggest here a unified approach to check the asymptotic compactness using the so-called energy method and approximations of a given external force by more regular ones. Mention also that although there is a number of papers where the existence of a strong attractors for equations with non translation compact external forces has been established, to the best of our knowledge, the usage of the energy method has been restricted to translation-compact external forces only.
\par
We start our exposition with 2D Navier-Stokes problem with non-autonomous external forces.

\subsection{Navier-Stokes equations} We consider the following Navier-Stokes problem in a bounded 2D domain $\Omega\subset\R^2$ with smooth boundary:
\begin{equation}\label{4.ns}
\Dt u+B(u,u)=\nu Au+g(t),\ \divv u=0,\ \ u\big|_{\partial\Omega}=0,\ \ \ u\big|_{t=\tau}=u_\tau.
\end{equation}
Here $A:=\Pi\Dx$ is a Stokes operator, $\nu>0$ is a kinematic viscosity, the nonlinearity
\begin{equation}
B(u,u):=\Pi((u,\Nx)u),
\end{equation}
where $\Pi$ is a Leray projector to divergent free vector fields and the external forces are assumed translation bounded in $H^{-1}(\Omega)$:
\begin{equation}\label{4.ext}
g\in L^2_b(\R,H^{-1}(\Omega)).
\end{equation}
It is well-known that equations \eqref{4.ns} are well-posed in the phase space $H$ which is the closure of divergent free vector fields $C_0^\infty(\Omega)$ in the topology of $[L^2(\Omega)]^2$. Namely, for any $\tau\in\R$ and every $u_\tau\in H$, there is a unique solution
\begin{equation}
u\in C([\tau,T],H)\cap L^2([\tau,T],H^1_0(\Omega)),
\end{equation}
for any $T\ge\tau$. Moreover, this solution satisfies the energy estimate
\begin{equation}\label{4.en}
\|u(t)\|_H^2\le \|u_\tau\|^2_H e^{-\beta(t-\tau)}+C\|g\|^2_{L^2_b(\R,H^{-1}(\Omega))},
\end{equation}
where $\beta,C>0$ are independent of $\tau$, $t$ and $u$.
\par
Following the general scheme, we introduce a weak hull $\Cal H(g)$ of the given external force $g$ and consider a family of Navier-Stokes problems:
\begin{equation}\label{4.ns-hull}
\Dt u+B(u,u)=\nu Au+h(t),\ \divv u=0,\ \ \ \ \ u\big|_{t=\tau}=u_\tau,\ \ h\in\Cal H(g).
\end{equation}
Then, these problems generate a family of dynamical processes $U_h(t,\tau):H\to H$, $h\in\Cal H(g)$ and estimate \eqref{4.en} guarantees that this family is uniformly dissipative. Moreover, the weak continuity of these processes is also straightforward and, therefore, according to Theorem \ref{Th3.wa}, there exists a {\it weak} uniform attractor $\Cal A$ of the problem \eqref{4.ns-hull} in $H$ which possesses the description \eqref{3.str}, see \cite{CVbook} for more details.
\par
Thus, the existence of a weak uniform attractor $\Cal A$ is well-known for the problem considered and our next task is to verify that this attractor is actually strong under the additional assumptions on the external forces $g$. As we know from Theorem \ref{Th3.strong}, it is sufficient to verify the asymptotic compactness to gain the strong attractor.
\par
We start with the case of normal external forces. Although the result of the next theorem is not new, see \cite{Lu}, the previous proofs require rather delicate and complicated technique and our proof is based on the elementary energy method.

\begin{theorem}\label{Th4.ns-norm}Let the above assumptions hold and let, in addition, the external force $g\in L^2_b(\R,H^{-1}(\Omega))$ be normal.
Then, the weak uniform attractor $\Cal A$ in $H$ related with the family of Navier-Stokes problems \eqref{4.ns-hull} is actually strong uniform attractor.
\end{theorem}
\begin{proof} According to Theorem \ref{Th3.strong}, we only need to verify the asymptotic compactness. Let $h_n\in\Cal H(g)$, $\tau_n\to-\infty$
and $u_{\tau_n}\in H$ be a bounded sequence. We need to check that the sequence $U_{h_n}(0,\tau_n)u_{\tau_n}$ is precompact in $H$, see Remark \ref{Rem3.as-comp}. Let $u_n(t)=U_{h_n}(t,\tau_n)u_{\tau_n}$ be the corresponding solutions. Then, these functions solve
\begin{equation}\label{4.ns-seq}
\Dt u_n+B(u_n,u_n)=\nu A u_n+h_n(t),\ \divv u_n=0,\ \ u_n\big|_{t=\tau_n}=u_{\tau_n},\ \ t\ge\tau_n.
\end{equation}
Moreover, since the sequence $u_{\tau_n}$ is bounded, the sequence of corresponding solutions $u_n(t)$ is bounded in $L^\infty_{loc}(\R,H)\cap L^2_{loc}(\R,H^1_0(\Omega))$ (we just extend $u_n(t)$ by zero for $t\le\tau_n$). Thus, without loss of generality, we may assume that the sequence of external forces $h_n\to h\in\Cal H(g)$ weakly in $L^2_{loc}(\R,H^{-1}(\Omega))$ and
$u_n\to u$ weakly-star in $L^\infty_{loc}(\R,H)$ and weakly in $L^2_{loc}(\R,H^1_0(\Omega))$. Moreover, estimating the time derivative $\Dt u_n$ from equations \eqref{4.ns-seq} and using the compactness theorems, we see that $u_n\to u$  {\it strongly} in $L^2_{loc}(\R,H)$,
 see \cite{CVbook} for the details. Furthermore, passing to the limit $n\to\infty$ in equations \eqref{4.ns-seq}, we see that the limit function $u(t)$ solves
\begin{equation}\label{4.ns-lim}
\Dt u+B(u,u)=\nu A u+h(t),\ \divv u=0,\ \ t\in\R
\end{equation}
and, therefore, $u\in\Cal K_h$. Finally, we have the weak convergence
\begin{equation}\label{4.weak}
u_n(0)\rightharpoondown u(0),\ \ \text{ in } \ H
\end{equation}
and the theorem will be proved if we check that the convergence in \eqref{4.weak} is actually {\it strong}. To this end, we will use the so-called energy method. Indeed, the energy equality for the solution $u_n$ reads
\begin{equation}\label{4.energy}
\frac d{dt}\|u_n\|^2_H+2\nu\|\Nx u_n\|^2_{L^2}=2(h_n,u_n).
\end{equation}
Here and below $(u,v)$ means the standard inner product in $L^2$. Introducing the artificial parameter $N>0$, we transform \eqref{4.energy} as follows
\begin{multline}
\frac d{dt}((t+1)\|u_n\|^2_H)+N(t+1)\|u_n\|^2_{H}+2\nu(t+1)\|\Nx u_n\|^2_{L^2}=\\=\|u_n\|^2_{H}+N(t+1)\|u_n\|^2_H+2(t+1)(h_n,u_n)
\end{multline}
and after the integration over $t\in[-1,0]$, we get
\begin{multline}\label{4.m-en}
\|u_n(0)\|^2_H+2\nu\int_{-1}^0e^{Ns}(s+1)\|\Nx u_n(s)\|^2_{L^2}\,ds=\\=\int_{-1}^0e^{Ns}(\|u_n(s)\|^2_H+N(s+1)\|u_n(s)\|^2_{H})\,ds+2\int_{-1}^0e^{Ns}(s+1)(h_n(s),u_n(s))\,ds.
\end{multline}
We want to pass to the limit $n\to\infty$ in \eqref{4.m-en}. Note that first and second terms in the left-hand side of \eqref{4.m-en} do not arise any problems since they are non-negatively definite quadratic forms and the weak limit there will not exceed $\liminf_{n\to\infty}$ of these terms. The first term in the right-hand side of \eqref{4.m-en} also arises no problems since we have the {\it strong} convergence $u_n\to u$ in $L^2((-1,0)\times\Omega)$. Thus, the main problem is to pass to the limit in the term involving the external forces $h_n$. Indeed, for this term, we only know that $u_n\rightharpoondown u$ in $L^2(-1,0;H^1_0(\Omega))$ and $h_n\rightharpoondown h$ in $L^2(-1,0;H^{-1}(\Omega))$, so the convergence of a product $(h_n,u_n)$ is not guaranteed.
\par
Instead of passing to the limit in this term, we use the assumption that $g$ is normal and show that it is {\it small} when $N$ is large. Indeed, this term can be estimated as follows:
\begin{multline}\label{4.small}
|\int_{-1}^0e^{Ns}(s+1)(h_n(s),u_n(s))\,ds|\le C_\eb\int_{-1}^0e^{2Ns}\|h_n(s)\|^2_{H^{-1}(\Omega)}\,ds+\\+\eb\int_{-1}^0\|u_n(s)\|^2_{H^1_0}\,ds\le
C_\eb \eb_N+C\eb\le C\eb_N,
\end{multline}
where the first term is small when $N$ is large due to the fact that $g, h_n$ are normal and Lemma \ref{Lem1.normalexp}, see also \eqref{2.hull-norm}, and the second term is small due to the choice of $\eb>0$ and boundedness of $u_n$. Thus, the term \eqref{4.small} is indeed small when the artificial parameter  $N$ is large. Then, passing to the limit $n\to\infty$ in \eqref{4.m-en}, we have
\begin{multline}\label{4.m-en-lim}
\limsup_{n\to\infty}\|u_n(0)\|^2_H+2\nu\int_{-1}^0e^{Ns}(s+1)\|\Nx u(s)\|^2_{L^2}\,ds\le\\\le\int_{-1}^0e^{Ns}(\|u(s)\|^2_H+N(s+1)\|u(s)\|^2_{H})\,ds+C\eb_N,
\end{multline}
where $\eb_N$ is small when $N$ is large. Writing now the analogue of \eqref{4.m-en} for the limit function $u$ and estimating the term containing $h$ analogously to \eqref{4.small}, we have
\begin{multline}\label{4.m-en-lim-lim}
\|u(0)\|^2_H+2\nu\int_{-1}^0e^{Ns}(s+1)\|\Nx u(s)\|^2_{L^2}\,ds\ge\\\ge\int_{-1}^0e^{Ns}(\|u(s)\|^2_H+N(s+1)\|u(s)\|^2_{H})\,ds-C\eb_N.
\end{multline}
Substructing \eqref{4.m-en-lim-lim} from \eqref{4.m-en-lim}, we get
\begin{equation}
\limsup_{n\to\infty}\|u_n(0)\|^2_H\le \|u(0)\|^2_H+2C\eb_N
\end{equation}
and, finally, passing to the limit $N\to\infty$, we arrive at
\begin{equation}\label{4.str}
\limsup_{n\to\infty}\|u_n(0)\|^2_H\le \|u(0)\|^2_H\le \liminf_{n\to\infty}\|u_n(0)\|_H^2,
\end{equation}
where the right inequality is due to the weak convergence $u_n(0)\rightharpoondown u(0)$. The inequalities \eqref{4.str} imply that
\begin{equation}
\lim_{n\to\infty}\|u_n(0)\|^2_H=\|u(0)\|^2_H
\end{equation}
which together with the weak convergence $u_n(0)\rightharpoondown u(0)$ implies the {\it strong} convergence $u_n(0)\to u(0)$ in $H$. Thus, the asymptotic compactness is verified and the theorem is proved.
\end{proof}
\begin{remark}\label{Rem4.w-normal} The proof given above can be easily extended to the case of weakly normal external forces. Indeed, in this case,
for every $\eb>0$, we may split the external forces $h_n$ on the functions $\varphi_n$ with finite dimensional range which can be assumed without loss of generality to be weakly convergent in $L^2_{loc}(\R,L^2(\Omega))$, so the convergence in the term containing $(\varphi_n,u_n)$ is straightforward and on the part $h_n-\varphi_n$  the integral
\begin{equation}
\int_{-1}^0e^{Ns}(s+1)(h_n(s)-\varphi_n(s),u_n(s))\,ds
\end{equation}
is {\it small} when $N\to\infty$ due to Lemma \ref{Lem1.w-normalexp}. Thus, passing to the limit $N\to\infty$ and then $\eb\to0$, we obtain \eqref{4.str} which gives the desired asymptotic compactness. Nevertheless, in order to illustrate how to work with space regular external forces, we give below the proof of asymptotic compactness for the particular case when $g$ is space regular.
\end{remark}

\begin{theorem}\label{Th4.ns-space}Let the above assumptions hold and let, in addition, the external force $g\in L^2_b(\R,H^{-1}(\Omega))$ be space regular.
Then, the weak uniform attractor $\Cal A$ in $H$ related with the family of Navier-Stokes problems \eqref{4.ns-hull} is actually strong uniform attractor.
\end{theorem}
\begin{proof} The beginning of the proof of this theorem repeats word by word the proof of the previous theorem. The first difference appear at \eqref{4.m-en}, where we need not the artificial parameter $N$ and may just take $N=0$:
\begin{multline}\label{4.us-en}
\|u_n(0)\|^2_H+2\nu\int_{-1}^0(s+1)\|\Nx u_n(s)\|^2_{L^2}\,ds=\\=\int_{-1}^0\|u_n(s)\|^2_H\,ds+2\int_{-1}^0(s+1)(h_n(s),u_n(s))\,ds.
\end{multline}
In this identity, the term containing $h_n$ is no more small and should be treated in a different way using the fact that $g$ is space regular. Namely, according to Corollary \ref{Cor2.space}, for every $\eb>0$ there exists a sequence $\varphi_n\in L^2_b(\R,L^2(\Omega))$ such that $
\varphi_n\to\varphi$ weakly in $L^2_{loc}(\R,L^2(\Omega))$ and
\begin{equation}\label{4.eb}
\|h_n-\varphi_n\|_{L^2_b(\R,H^{-1}(\Omega))}+\|h-\varphi\|_{L^2_b(\R,H^{-1}(\Omega))}\le\eb.
\end{equation}
Thus,
\begin{multline}
|\int_{-1}^0(s+1)(h_n(s),u_n(s))\,ds-\int_{-1}^0(s+1)(\varphi_n(s),u_n(s))\,ds|\le\\\le C\|h_n-\varphi_n\|_{L^2_b(\R,H^{-1}(\Omega))}\|u_n\|_{L^2(-1,0;H^1_0(\Omega))}\le C\eb
\end{multline}
and the analogous estimate holds for the limit case $h$ and $u$ as well. Therefore, with accuracy $C\eb$ we may replace the external forcers $h_n$
by $\varphi_n$. On the other hand, we know that $u_n\to u$ {\it strongly} in $L^2(-1,0;L^2(\Omega))$, so
\begin{equation}\label{4.conv}
\int_{-1}^0(s+1)(\varphi_n(s),u_n(s))\,ds\to \int_{-1}^0(s+1)(\varphi(s),u(s))\,ds
\end{equation}
as $n\to\infty$. Passing now to the limit $n\to\infty$ in \eqref{4.us-en}, we get
\begin{multline}
\limsup_{n\to\infty}\|u_n(0)\|^2_H+2\nu\int_{-1}^0(s+1)\|\Nx u(s)\|^2_{L^2}\,ds\le\\\le\int_{-1}^0\|u(s)\|^2_H\,ds+2\int_{-1}^0(s+1)(\varphi(s),u(s))\,ds+C\eb.
\end{multline}
Writing now the analogue of \eqref{4.us-en} for the limit function $u$ and replacing $h$ by $\varphi$ in it, we have
\begin{multline}
\|u(0)\|^2_H+2\nu\int_{-1}^0(s+1)\|\Nx u(s)\|^2_{L^2}\,ds\ge\\\ge\int_{-1}^0\|u(s)\|^2_H\,ds+2\int_{-1}^0(s+1)(\varphi(s),u(s))\,ds-C\eb
\end{multline}
which gives
\begin{equation}
\limsup_{n\to\infty}\|u_n(0)\|^2_H\le \|u(0)\|^2_H+2C\eb.
\end{equation}
Since $\eb>0$ is arbitrary, passing to the limit $\eb\to0$, we get \eqref{4.str} which implies the strong convergence $u_n(0)\to u(0)$ in $H$, see the end of the proof of Theorem \ref{Th4.ns-norm}. Thus, the asymptotic compactness is proved and the theorem is also proved.
\end{proof}

\subsection{Damped wave equation} Our next example will be the so-called damped wave equation which is not parabolic and normality of the external forces is not enough to gain the asymptotic compactness. However, as we will see space regularity or time regularity of the external forces is sufficient for the desired compactness. The main method of verifying the asymptotic compactness will be again the energy method.
\par
Let us consider the following damped wave equation in a bounded domain $\Omega$ of $\R^3$ with the smooth boundary:
\begin{equation}\label{4.hyp}
\Dt^2 u+\gamma\Dt u-\Dx u+f(u)=g(t),\ u\big|_{\partial\Omega}=0,\ (u,\Dt u)\big|_{t=\tau}=(u_\tau,u_\tau').
\end{equation}
Here $\gamma>0$ is a fixed dissipation rate, $g(t)$ are given non-autonomous external forces which are assumed translation bounded in $L^2(\Omega)$:
\begin{equation}\label{4.h-ext}
g\in L^2_b(\R,L^2(\Omega))
\end{equation}
and, finally the non-linear interaction function $f$ is assumed to satisfy some natural dissipativity and growth restrictions, namely,
\begin{equation}\label{4.f}
1.\ f(u)u\ge -C;\ \ 2.\ \ |f''(u)|\le C(1+|u|).
\end{equation}
Note that according to the recent results of \cite{KSZ}, the cubic growth restriction can be relaxed and the {\it quintic} growth rate of $f$ is now considered as the critical one. However, in order to avoid the technicalities, we pose here the "old" cubic growth restrictions.
\par
It is well-known, see e.g., \cite{CVbook}, that under the above assumptions problem \eqref{4.hyp} is globally well-posed in the energy phase space
\begin{equation}\label{4.h-energy}
E:=H^1_0(\Omega)\times L^2(\Omega),\ \ \xi_u(t):=(u(t),\Dt u(t))\in E,\ \ t\ge\tau.
\end{equation}
Thus, for every $\xi_\tau\in E$, there exists a unique solution $u(t)$, $t\ge\tau$, satisfying
\begin{equation}\label{4.h-def}
\xi_u\in C(\tau,T;E),\ \ \text{for any}\ \ T\ge\tau,\ \ \xi_u(\tau)=\xi_\tau.
\end{equation}
Moreover, this solution possesses the following dissipative estimate:
\begin{equation}\label{4.h-dis}
\|\xi_u(t)\|_E\le Q(\|\xi_\tau\|_E)e^{-\beta(t-\tau)}+Q(\|g\|_{L^2_b(\R,L^2(\Omega))}),
\end{equation}
where the positive constant $\beta$ and monotone function $Q$ are independent of $\tau\in\R$, $t\ge\tau$ and $u$, see e.g., \cite{CVbook}. As before, following the general scheme of Vishik and Chepyzhov, we consider a family of damped wave equations with the external forces belonging to the hull $\Cal H(g)$ of the initial right-hand side $g(t)$:
\begin{equation}\label{4.hyp-hull}
\Dt^2 u+\gamma\Dt u-\Dx u+f(u)=h(t),\ u\big|_{\partial\Omega}=0,\ \xi_u\big|_{t=\tau}=\xi_\tau,\ h\in\Cal H(g).
\end{equation}
Then, a family of dynamical processes $U_h(t,\tau): E\to E$, $h\in\Cal H(g)$, is well-defined. Moreover, according to estimate \eqref{4.h-dis} this family is uniformly dissipative. The weak continuity is also straightforward here and, therefore, due to Theorem \ref{Th3.wa}, there exist a weak uniform attractor $\Cal A\subset E$ for the family of processes generated by \eqref{4.hyp-hull} and this attractor obeys \eqref{3.str}, see \cite{CVbook} for a more detailed exposition.
\par
Thus, as in the previous example, the existence of a weak uniform attractor is well-known and straightforward here and our task is to verify that it is actually a {\it strong} uniform attractor (based on Theorem \ref{Th3.strong}) under some extra assumptions on the external force $g(t)$. We start with the case when $g$ is time regular.

\begin{theorem}\label{Th4.hyp-time}Let the above assumptions hold and let, in addition, the external force $g\in L^2_b(\R,L^2(\Omega))$ be time regular.
Then, the weak uniform attractor $\Cal A$ in $E$ related with the family of damped wave  problems \eqref{4.hyp-hull} is actually strong uniform attractor.
\end{theorem}
\begin{proof} As in the case of Neavier-Stokes equations, we only need to check the asymptotic compactness. Let $\tau_n\to-\infty$, $h_n\in\Cal H(g)$, $\xi_{\tau_n}\in E$ be bounded in $E$ and let $\xi_{u_n}(t):= U_{h_n}(t,\tau_n)\xi_{\tau_n}$ be the corresponding solutions. Then, the functions $u_n(t)$ solve
\begin{equation}\label{4.h-xun}
\Dt^2 u_n+\gamma\Dt u_n-\Dx u_n+f(u_n)=h_n(t),\ \ t\ge\tau_n,\ \ \xi_{u_n}\big|_{t=\tau_n}=\xi_{\tau_n}.
\end{equation}
Due to the energy estimate \eqref{4.h-dis}, the sequence $u_n$ is bounded in $L^\infty(\R,E)$ (we again extend $\xi_{u_n}$ by zero for $t\le\tau_n$). Thus, without loss of generality, we may assume that $u_n\to u$ weakly-star in $L^\infty_{loc,w}(\R,E)$. Moreover, estimating the second time derivative from equations \eqref{4.h-xun} and using the proper compactness theorem, we get the {\it strong} convergence
\begin{equation}\label{40.h-st-conv}
u_n\to u,\ \ \text{strongly in}\  \ C_{loc}(\R,E_{-\delta}),\ \delta>0,
\end{equation}
where $E_{-\delta}:=H^{1-\delta}_0(\Omega)\cap H^{-\delta}(\Omega)$, see \cite{CVbook} for more details.
\par
Furthermore, without loss of generality, we may assume that the external forces $h_n\rightharpoondown h$ in $L^2_{loc}(\R,L^2(\Omega))$.  Passing now to the limit $n\to\infty$ in equations \eqref{4.h-xun} (the strong convergence \eqref{40.h-st-conv} together with the growth restriction \eqref{4.f} allows us to pass to the limit in  the non-linear term $f(u_n)$ in a straightforward way), we see that the limit function $u\in \Cal K_{h}$ solves
\begin{equation}\label{4.h-un-lim}
\Dt^2 u+\gamma\Dt u-\Dx u+f(u)=h(t),\ \ t\in\R.
\end{equation}
Finally, from the above convergences, we also have
\begin{equation}
\xi_{u_n}(0)\rightharpoondown \xi_u(0) \ \ \text{in}\ \ E
\end{equation}
and to finish the proof of the theorem, we only need to show that the convergence is actually {\it strong}. To this end
we will again use the energy method.
\par
Remind that the energy equality for \eqref{4.h-xun} reads
\begin{equation}\label{4.h-en1}
\frac d{dt}\(\|\xi_{u_n}\|_E^2+2(F(u_n),1)\)+2\gamma\|\Dt u\|^2_{L^2}=2(h_n,\Dt u_n),
\end{equation}
where $F(v)=\int_0^vf(u)\,du$. Moreover, multiplying \eqref{4.h-xun} by $u$ and inegrating over $x$, we have
\begin{equation}\label{4.h-en2}
\frac d{dt}\((u_n,\Dt u_n)+\frac\gamma2\|u_n\|^2\)+\|\Nx u_n\|^2_{L^2}-\|\Dt u_n\|^2_{L^2}+(f(u_n),u_n)=(h_n,u_n).
\end{equation}
Multiplying \eqref{4.h-en2} by $\gamma$ and taking a sum with \eqref{4.h-en1}, we end up with
\begin{multline}\label{4.h-en3}
\frac{d}{dt}\(\|\xi_{u_n}\|^2_E+2(F(u_n),1)+\gamma(u_n,\Dt u_n)+\frac{\gamma^2}2\|u_n\|^2_{L^2}\)+\gamma\|\xi_{u_n}\|^2_E+\\+\gamma(f(u_n),u_n)=(h_n,2\Dt u_n+\gamma u_n).
\end{multline}
We rewrite this identity in the form
\begin{equation}\label{4.en-4}
\frac d{dt} E(u_n(t))+\gamma E(u_n(t))+\gamma G(u_n)=(h_n,2\Dt u_n+\gamma u_n),
\end{equation}
where $E(u_n):=\|\xi_{u_n}\|^2_E+2(F(u_n),1)+\gamma(u_n,\Dt u_n)+\frac{\gamma^2}2\|u_n\|^2_{L^2}$ and $G(u_n):=(f(u_n),u_n)-2(F(u_n),1)-\gamma(u_n,\Dt u_n)-\frac{\gamma^2}2\|u_n\|^2_{L^2}$. Integrating the last equality over $t\in(\tau_n,0)$, we finally get
\begin{multline}\label{4.en-5}
\|\xi_{u_n}(0)\|^2_E+2(F(u_n(0)),1)+\gamma(u_n(0),\Dt u_n(0))+\frac{\gamma^2}2\|u_n(0)\|^2_{L^2}+\\+\gamma\int_{\tau_n}^0e^{\gamma s}G(u_n(s))\,ds=
E(u_{n}(\tau_n))e^{\gamma\tau_n}+\int_{\tau_n}^0e^{\gamma s}(h_n(s),2\Dt u_n(s)+\gamma u_n(s))\,ds.
\end{multline}
We intend to pass to the limit $n\to\infty$ in this identity. To this end, we note that due to uniform boundedness of $\xi_{u_n}\in C_b(\tau_n,0;E)$ and {\it strong} convergence \eqref{40.h-st-conv}, passing to the limit in all terms of \eqref{4.en-5} except of the one containing $(h_n,\Dt u_n)$ is straightforward. Thus, we only need to show that
\begin{equation}\label{4.h-hconv}
\int_{\tau_n}^0e^{\gamma s}(h_n(s),\Dt u_n(s))\,ds\to\int_{-\infty}^0e^{\gamma s}(h(s),\Dt u(s))\,ds.
\end{equation}
To verify \eqref{4.h-hconv}, we utilize the fact that the external force $g\in L^2_b(\R,L^2(\Omega))$ is time regular. Then, according to Corollary \ref{Cor2.time}, for any $\eb>0$, there exists a sequence $\varphi_n\in H^1_b(\R,L^2(\Omega))$ which is uniformly bounded in this space such that $\varphi_n\rightharpoondown\varphi$ in $H^1_{loc}(\R,L^2(\Omega))$ and
\begin{equation}\label{4.h-app}
\|h_n-\varphi_n\|_{L^2_b(\R,L^2(\Omega))}+\|h-\varphi\|_{L^2_b(\R,L^2(\Omega))}\le \eb.
\end{equation}
This, inequality together with uniform boundedness $\Dt u_n\in C_b(\tau_n,0;L^2(\Omega))$,implies that
\begin{multline}\label{4.h-app1}
|\int_{\tau_n}^0e^{\gamma s}(h_n(s),\Dt u_n(s))\,ds-\int_{\tau_n}^0e^{\gamma s}(\varphi_n(s),\Dt u_n(s))\,ds|+\\+|\int_{-\infty}^0e^{\gamma s}(h(s),\Dt u(s))\,ds-\int_{-\infty}^0e^{\gamma s}(\varphi(s),\Dt u(s))\,ds|\le C\eb.
\end{multline}
Thus, with accuracy $C\eb$, we may replace $h_n$ by $\varphi_n$. Since $\varphi_n\in H^1_b(\R,L^2(\Omega))$, we may integrate by parts
\begin{multline}
\int_{\tau_n}^0e^{\gamma s}(\varphi_n(s),\Dt u_n(s))\,ds=-\int_{\tau_n}^0(\Dt \varphi_n(s),u_n(s))\,ds-\\-\gamma\int_{\tau_n}^0e^{\gamma s}(\varphi_n(s),u_n(s))\,ds+(\varphi_n(0),u_n(0))-e^{\gamma\tau_n}(\varphi_n(\tau_n),u_n(\tau_n)).
\end{multline}
Using now the fact that $\varphi_n\to\varphi$ weakly in $H^1_{loc}(\R,L^2(\Omega))$ and, particularly, $\varphi_n(0)\rightharpoondown \varphi(0)$ together with the strong convergence \eqref{40.h-st-conv}, we see that
\begin{equation}
\int_{\tau_n}^0e^{\gamma s}(\varphi_n(s),\Dt u_n(s))\,ds\to \int_{-\infty}^0e^{\gamma s}(\varphi(s),\Dt u(s))\,ds
\end{equation}
as $n\to\infty$. Since $\eb>0$ is arbitrary, this convergence, together with estimate \eqref{4.h-app1} implies the desired convergence \eqref{4.h-hconv}. Passing now to the limit $n\to\infty$ in \eqref{4.en-5}, we end up with
\begin{multline}\label{4.en-6}
\limsup_{n\to\infty}\|\xi_{u_n}(0)\|^2_E+2(F(u(0)),1)+\gamma(u(0),\Dt u(0))+\frac{\gamma^2}2\|u(0)\|^2_{L^2}+\\+\gamma\int_{-\infty}^0e^{\gamma s}G(u(s))\,ds\le\int_{-\infty}^0e^{\gamma s}(h(s),2\Dt u(s)+\gamma u(s))\,ds.
\end{multline}
Writing out now the analogue of \eqref{4.en-5} for the limit function $u$, we have
\begin{multline}\label{4.en-7}
\|\xi_{u}(0)\|^2_E+2(F(u(0)),1)+\gamma(u(0),\Dt u(0))+\frac{\gamma^2}2\|u(0)\|^2_{L^2}+\\+\gamma\int_{-\infty}^0e^{\gamma s}G(u(s))\,ds=\int_{-\infty}^0e^{\gamma s}(h(s),2\Dt u(s)+\gamma u(s))\,ds.
\end{multline}
Thus,
\begin{equation}\label{4.h-fincon}
\limsup_{n\to\infty}\|\xi_{u_n}(0)\|^2_E\le \|\xi_u(0)\|^2_E\le\liminf_{n\to\infty}\|\xi_{u_n}(0)\|^2_E,
\end{equation}
where the right inequality comes from the weak convergence $\xi_{u_n}(0)\to\xi_{u}(0)$ in $E$. Inequalities \eqref{4.h-fincon} implly that
$\lim_{n\to\infty}\|\xi_{u_n}(0)\|^2_E=\|\xi_u(0)\|^2_E$ which together with the weak convergence imply the desired {\it strong} convergence $\xi_{u_n}(0)\to\xi_u(0)$ in $E$. Thus, the asymptotic compactness is proved and the theorem is also proved.
\end{proof}
We now state the analogous result for the case when the external forces are {\it space} regular.
\begin{theorem}\label{Th4.hyp-space}Let the above assumptions hold and let, in addition, the external force $g\in L^2_b(\R,L^2(\Omega))$ be space regular.
Then, the weak uniform attractor $\Cal A$ in $E$ related with the family of damped wave  problems \eqref{4.hyp-hull} is actually strong uniform attractor.
\end{theorem}
\begin{proof} The proof of this theorem repeats word by word the proof of the previous one. The only difference is that the convergence \eqref{4.h-hconv} should be now proved in a bit different way using that $g$ is space regular. Namely, according to Corollary \ref{Cor2.space}, for any $\eb>0$, there exist fuctions $\varphi_n\in L^2_b(\R,H^1(\Omega))$ (uniformly bounded in this space) such that $\varphi_n\to\varphi$ weakly in $L^2_{loc}(\R,H^1(\Omega))$ such that estimate \eqref{4.h-app} holds. Then, exactly as in the proof of the previous theorem, estimate \eqref{4.h-app1} also holds, so we may replace $h_n$ by $\varphi_n$. Since $\varphi_n$ is more regular in space, we may write for sufficiently small $\delta>0$
\begin{multline}
\int_{\tau_n}^0(\varphi_n(s),\Dt u_n(s))\,ds=\int_{\tau_n}^0((-\Dx)^\delta \varphi_n(s),(-\Dx)^{-\delta}\Dt u_n(s))\,ds\to\\\to\int_{-\infty}^0((-\Dx)^\delta \varphi(s),(-\Dx)^{-\delta}\Dt u(s))\,ds=\int_{-\infty}^0(\varphi(s),\Dt u(s))\,ds
\end{multline}
as $n\to\infty$. Here we have used that $\Dt u_n\to\Dt u$ {\it strongly} in $C_{loc}(\R,H^{-2\delta}(\Omega))$ and $\varphi_n\to\varphi$ weakly in
$L^2_{loc}(\R,H^{2\delta}(\Omega))$. Since $\eb>0$ is arbitrary, this prove the convergence \eqref{4.h-hconv}. The rest of the proof of the theorem coincides with the proof of the previous theorem.
\end{proof}

\subsection{Reaction-diffusion equation in unbounded domain} Our last example will be related with reaction-diffusion system in $\Omega=\R^N$ in the class of finite energy solutions. As not difficult to see, time regularity of the external forces is not enough to get strong uniform attractor, e.g. since time regularity does not exclude external forces in the form of travelling waves $g(t,x)=g(t-\alpha x)$ for which the uniform attractor cannot exist. However, as we will see, the {\it space} regularity of the external forces is still enough to have a strong uniform attractor.
\par
Let us consider the following reaction-diffusion system in $\R^N$:
\begin{equation}\label{4.r-eq}
\Dt u=a\Dx u-\alpha u-f(u)+g(t),\ \ u\big|_{t=\tau}=u_\tau.
\end{equation}
Here $u=(u^1,\cdots, u^k)$ is an unknown vector-valued function $a$ is a given diffusion matrix which satisfies $a+a^*>0$, $\alpha>0$ is a given parameter, $g(t)$ are the non-autonomous external forces which are translation bounded in $H^{-1}(\Omega)$:
\begin{equation}\label{4.r-h}
g\in L^2_b(\R,H^{-1}(\Omega)).
\end{equation}
Finally, the nonlinear interaction function $f$ is assumed to satisfy the following assumptions:
\begin{equation}\label{4.r-f}
\begin{cases}
1.\ \ f\in C^1(\R^k,\R^k);\\
2.\ \ f(u).u\ge \beta |u|^p;\\
3.\ \ f'(u)\ge -K;\\
4.\ \ |f(u)|\le C(1+|u|^{p-1})
\end{cases}
\end{equation}
for some positive $\beta$ and $p>1$.
\par
It is well-known, see \cite{MZ} and references therein that under the above assumptions problem \eqref{4.r-eq} is globally well-posed in the phase space $L^2(\R^N)$, i.e., for every $u_\tau\in L^2(\R^N)$ there exists a unique solution $u(t)$, $t\ge\tau$ belonging to the class
\begin{equation}\label{4.r-sol}
u\in C(\tau,T;L^2(\R^N))\cap L^2(\tau,T;H^1_0(\R^N))\cap L^p(\tau,T;L^p(\R^N)),\ \ \forall T\ge\tau.
\end{equation}
Important for us is the fact that this solution satisfies the energy {\it identity}
\begin{equation}\label{4.r-en}
\frac12 \frac d{dt}\|u\|^2_{L^2}+\alpha\|u\|^2_{L^2}+(a\Nx u,\Nx u)+(f(u),u)=(g,u).
\end{equation}
In particular, this energy identity together with our assumptions on $f$ and $a$ gives the dissipative energy estimate:
\begin{multline}\label{4.r-dis}
\|u(t)\|^2_{L^2}+\int_{t}^{t+1}\|u(s)\|^2_{H^1}\,ds+\int_{t}^{t+1}\|u(s)\|^p_{L^p}\,ds\le\\\le Ce^{-\gamma (t-\tau)}\|u_\tau\|_{L^2}^2+C\|g\|^2_{L^2_b(\R,H^{-1}(\R^N))}.
\end{multline}
Following the general procedure, we consider the family of problems of type \eqref{4.r-eq}
\begin{equation}\label{4.r-eq1}
\Dt u=a\Dx u-\alpha u-f(u)+h(t),\ \ u\big|_{t=\tau}=u_\tau,\ \ h\in\Cal H(g).
\end{equation}
Then, the family of dynamical processes $U_h(t,\tau):L^2(\R^N)\to L^2(\R^N)$, $h\in\Cal H(g)$. related with problems \eqref{4.r-eq1} is well-defined.
Moreover, due to estimate \eqref{4.r-dis}, this family is uniformly dissipative. As usual, weak continuity of this processes can be verified in a straightforward way, so we left the rigorous proof of it to the reader.
\par
Thus, due to Theorem \ref{Th3.wa}, the family of processes $U_h(t,\tau):L^2(\R^N)\to L^2(\R^N)$, $h\in\Cal H(g)$, associated with problems \eqref{4.r-eq1} possess a weak uniform attractor $\Cal A$ in $L^2(\R^N)$. Our next task, as usual, to verify that under some extra assumptions on the external forces $g$ this attractor is strong.

\begin{theorem}\label{Th4.r-space}Let the above assumptions hold and let, in addition, the external force $g\in L^2_b(\R,H^{-1}(\R^N))$ be space regular.
Then, the weak uniform attractor $\Cal A$ in $L^2(\R^N)$ related with the family of damped wave  problems \eqref{4.r-eq1} is actually strong uniform attractor.
\end{theorem}
\begin{proof} As before, we will use the energy method to verify the asymptotic compactness. Let $\tau_n\to-\infty$, $h_n\in\Cal H(g)$, and $u_{\tau_n}$ be a bounded sequence in $L^2(\R^N)$. Let also $u_n(t):=U_{h_n}(t,\tau_n)u_{\tau_n}$ be the corresponding solutions:
\begin{equation}\label{4.r-eq2}
\Dt u_n=a\Dx u_n-\alpha u_n-f(u_n)+h_n(t),\ \ u\big|_{t=\tau_n}=u_{\tau_n},\ \ t\ge\tau_n.
\end{equation}
First, we need to pass to the limit $n\to\infty$ in these equations. To this end, we note that, due to the energy estimate, the sequence $u_n$ is bounded in $L^\infty(\R,L^2(\R^N))\cap L^2_b(\R,H^1(\R^N))\cap L^p_b(\R,L^p(\R^N))$. Thus, without loss of generality, we may assume that
\begin{multline}\label{4.r-wconv}
u_n\to u\  \text{ weakly in }\ L^2_{loc}(\R,H^1(\R^N))\cap L^p_{loc}(\R,L^p(\R^N))\\ \text{ and weakly-star  in }\ L^\infty_{loc}(\R,L^2(\R^N))
\end{multline}
and $u_n(0)\rightharpoondown u(0)$ in $L^2(\R^N)$.
Moreover, estimating $\Dt u_n$ from the equation, we see that $\Dt u_n$ is bounded in $L^2_b(\R,H^{-1}(\R^N))+L^q_b(\R,L^q(\R^N))$ where $\frac1p+\frac1q=1$. However, in contrast to the case of bounded domains, this {\it does not} give the strong convergence in $L^2_{loc}(\R,L^2(\R^N))$ since the embedding $H^1(\R^N)\subset L^2(\R^N)$ is not compact, but only the strong convergence $u_n\to$ in the space $L^2_{loc}(\R,L^2_{loc}(\R^N))$ and, in particular, $u_n\to u$ almost everywhere. Finally, without loss of generality, we may assume that $h_n\to h\in\Cal H(g)$ weakly in $L^2_{loc}(\R,H^{-1}(\R^N))$. The established convergence is enough to pass to the limit $n\to\infty$ in equations \eqref{4.r-eq2} in a standard way and establish that the limit function $u\in\Cal K_h$ and solves
\begin{equation}\label{4.r-eq3}
\Dt u=a\Dx u-\alpha u-f(u)+h(t),\  \ t\in\R.
\end{equation}
Following the general procedure, we now write the energy equality for solutions $u_n(t)$ in the following integrated form:
\begin{multline}\label{4.r-inten}
\|u_n(0)\|^2_{L^2}+2\int_{\tau_n}^0e^{2\alpha s}(a\Nx u_n(s),\Nx u_n(s))\,ds+\\+2\int_{\tau_n}^0e^{2\alpha s}(f(u_n(s)),u_n(s))\,ds=\|u_{\tau_n}\|^2_{L^2}e^{2\alpha\tau_n}+2\int_{\tau_n}^0e^{2\alpha s}(h_n(s),u_n(s))\,ds.
\end{multline}

Our aim is to pass to the limit $n\to\infty$ in this energy equality. To this end, we note that, since $a$ is positive definite,
\begin{equation}\label{4.r-weak}
\int_{-\infty}^0e^{2\alpha s}(a\Nx u(s),\Nx u(s))\,ds\le \liminf_{n\to\infty}\int_{\tau_n}^0e^{2\alpha s}(a\Nx u_n(s),\Nx u_n(s))\,ds.
\end{equation}
Furthermore, using that $f(u_n)u_n\ge0$, the convergence $u_n\to u$ almost everywhere and the Fatou lemma, we get
\begin{equation}\label{4.r-weak1}
\int_{-\infty}^0e^{2\alpha s}(f(u(s)),u(s))\,ds\le \liminf_{n\to\infty}\int_{\tau_n}^0e^{2\alpha s}(f( u_n(s)), u_n(s))\,ds.
\end{equation}
Thus, it remains to pass to the limit $n\to\infty$ in the term containing $h_n$ and to prove that
\begin{equation}\label{4.r-weak2}
\int_{\tau_n}^0e^{2\alpha s}(h_n(s),u_n(s))\,ds\to\int_{-\infty}^0e^{2\alpha s}(h(s), u(s))\,ds
\end{equation}
as $n\to\infty$. To this end, we need to use the assumption that the external force $g\in L^2_b(\R,H^{-1}(\Omega))$ is space regular. Then, due to Corollary \ref{Cor2.space}, for every $\eb>0$ there exist $\varphi_n$ uniformly bounded in $L^2_b(\R,V_\eb)$ where $V_\eb\subset L^2(\R^N)$ is a finite-dimensional subspace of $H^{-1}(\R^N)$ such that $\varphi_n\to\varphi$ weakly in $L^2_{loc}(\R,L^2(\R^N))$ and
\begin{equation}\label{4.r-est1}
\|h_n-\varphi_n\|_{L^2_b(\R,H^{-1}(\R^N))}+\|h-\varphi\|_{L^2_b(\R,H^{-1}(\R^N)}\le\eb.
\end{equation}
This estimate shows that to prove \eqref{4.r-weak2}, it is sufficient to verify that
\begin{equation}\label{4.r-weak3}
\int_{\tau_n}^0e^{2\alpha s}(\varphi_n(s),u_n(s))\,ds\to\int_{-\infty}^0e^{2\alpha s}(\varphi(s), u(s))\,ds.
\end{equation}
However, in contrast to the case of bounded domains, the weak convergence $\varphi_n\to\varphi$ in $L^2_{loc}(\R,L^2(\R^N))$ is not sufficient to prove \eqref{4.r-weak3} since we have the strong convergence $u_n\to u$ only in $L^2_{loc}(\R,L^2_{loc}(\R^n))$ (and not in $L^2_{loc}(\R,L^2(\R^N))$). So, we need to utilize in addition the fact that the range of $\varphi_n$ belongd to the {\it finite-dimensional} space $V_\eb\subset L^2(\R^N)$. Let $\theta_1,\cdots, \theta_l\in L^2(\R^N)$ be a base in $V_\eb$. Then, on the one hand,
\begin{equation}\label{4.r-norm}
\text{for every } \varphi\in V_\eb,\ \ \varphi=\sum_{i=1}^l\xi_l\theta_l(x),\ \ \|\varphi\|_{L^2}\sim |\xi|.
\end{equation}
On the other hand, every $\theta_l$ satisfies the so-called tail estimate
\begin{equation}\label{4.r.tail}
\lim_{R\to\infty}\|\theta_l\|_{L^2(|x|>R)}=0.
\end{equation}
Since $V_\eb$ is finite dimensional, this give the so-called {\it uniform} tail estimate for the functions $\varphi_n$:
\begin{equation}\label{4.r.utail}
\lim_{R\to\infty}\|\varphi_n\|_{L^2_b(\R,L^2(|x|>R))}=0.
\end{equation}
Due to this tail estimate, the verification of the convergence \eqref{4.r-weak3} is reduced to the following one
\begin{equation}\label{4.r-weak4}
\int_{\tau_n}^0e^{2\alpha s}(\varphi_n(s),u_n(s))_{L^2(|x|\le R)}\,ds\to\int_{-\infty}^0e^{2\alpha s}(\varphi(s), u(s))_{L^2(|x|\le R)}\,ds
\end{equation}
for any $R>0$. But this convergence is obvious since $\varphi_n\rightharpoondown\varphi$  in $L^2_{loc}(\R,L^2(|x|\le R))$ and $u_n\to u$ strongly in $L^2_{loc}(\R,L^2(|x|\le R))$. Thus the desired convergence \eqref{4.r-weak2} is verified. Passing now to the limit $n\to\infty$ in energy equality
\eqref{4.r-inten} and using \eqref{4.r-weak2} together with \eqref{4.r-weak} and \eqref{4.r-weak1}, we have
\begin{multline}\label{4.r-inten1}
\limsup_{n\to\infty}\|u_n(0)\|^2_{L^2}+2\int_{-\infty}^0e^{2\alpha s}(a\Nx u(s),\Nx u(s))\,ds+\\+2\int_{-\infty}^0e^{2\alpha s}(f(u(s)),u(s))\,ds\le 2\int_{\tau_n}^0e^{2\alpha s}(h(s),u(s))\,ds.
\end{multline}
Writing down the analogue of energy equality \eqref{4.r-inten} for the limit function $u$, we have
\begin{multline}\label{4.r-inten2}
\|u(0)\|^2_{L^2}+2\int_{-\infty}^0e^{2\alpha s}(a\Nx u(s),\Nx u(s))\,ds+\\+2\int_{-\infty}^0e^{2\alpha s}(f(u(s)),u(s))\,ds= 2\int_{\tau_n}^0e^{2\alpha s}(h(s),u(s))\,ds
\end{multline}
and, therefore,
\begin{equation}
\limsup_{n\to\infty}\|u_n(0)\|^2_{L^2(\R^N)}\le \|u(0)\|^2_{L^2(\R^N)}\le \liminf_{n\to\infty}\|u_n(0)\|^2_{L^2(\R^N)}.
\end{equation}
Thus, $\lim_{n\to\infty}\|u_n(0)\|^2_{L^2(\R^N)}=\|u(0)\|^2_{L^2(\R^N)}$ which together with the weak convergence $u_n(0)\rightharpoondown u(0)$ in $L^2(\R^N)$ implies the desired strong convergence $u_n(0)\to u(0)$ in $L^2(\R^n)$ which proves the asymptotic compactness and finishes the proof of the theorem.
\end{proof}
\begin{remark}\label{Rem4.norm} As we have already pointed out, the normality of the external forces $g$ is not sufficient to verify the existence of a strong attractor in the case of unbounded domains. The main obstruction for that is the absence of the uniform tail estimate for the solutions $u$. However, if we assume, in addition, that the external forces $g$ possess the uniform tail estimate in the form
\begin{equation}\label{4.rtail}
\lim_{R\to\infty}\|g\|_{L^2_b(\R,H^{-1}(|x|>R))}=0,
\end{equation}
then the standard weighted energy estimates show that there is the analogous uniform tail estimate on the weak attractor $\Cal A$, see \cite{MZ} and references therein:
\begin{equation}
\lim_{R\to\infty}\sup_{h\in\Cal H(g)}\sup_{u\in\Cal K_h}\|u\|_{L^2_b(\R,H^1(|x|>R))}=0.
\end{equation}
Then, if $g$ is normal, we can obtain the desired asymptotic compactness via the energy method using the trick with the artificial parameter $N$, described in Theorem \ref{Th4.ns-norm}. Thus, normality of the external forces $g$ plus the uniform tail estimate \eqref{4.rtail} is enough to gain the asymptotic compactness and verify the existence of a strong uniform attractor.
\end{remark}
\section{Counterexamples}\label{s5} In this section, we give several examples showing that some  assumptions on the external forces $g$ which look natural and similar to the ones introduced before  are nevertheless {\it insufficient} to gain the asymptotic compactness. We start with the case where $g$ is only translation bounded.
\begin{example}\label{Ex5.par} Let us consider the 1D linear heat equation on $\Omega:=(-\pi,\pi)$ endowed by the Dirichlet boundary conditions
\begin{equation}\label{5.linear-heat}
\Dt u-\partial_x^2 u=g(t), \ \ u(0)=0,
\end{equation}
where the external force $g$ is given by the following explicit formula:
\begin{equation}\label{5.g}
g(t)=
\begin{cases}
n\sin(nx),\ \ t\in[n,n+\frac1{n^2}],\ n\in\Bbb N;\\
0,\ \ \text{otherwise.}
\end{cases}
\end{equation}
Then, on the one hand, the function $g$ is translation bounded in $L^2(\Omega)$: $g\in L^2_b(\R,L^2(\Omega))$, so we have the weak uniform attractor
for the linear problem \eqref{5.linear-heat} in $H^1_0(\Omega)$. On the other hand,
 the explicit computation shows that
\begin{equation}\label{5.bad}
u_n(n+\frac1{n^2})=\frac{1-e^{-1}}{n},
\end{equation}
where $u(t)=\sum_{n=1}^\infty u_n(t)\sin(nx)$. Identity \eqref{5.bad} shows that the sequence $u(n+\frac1{n^2})$, $n\in\Bbb N$, cannot be precompact in $H^1_0(\Omega)$, so the asymptotic compactness fails and equation \eqref{5.linear-heat} does not possess a strong uniform attractor in $H^1_0(\Omega)$.
\end{example}
\begin{remark}\label{Rem5.space-normal} The function $g$ defined by \eqref{5.g} is actually more regular, namely,
\begin{equation}\label{5.g-reg}
g\in L^2_b(\R,L^\infty(\Omega)).
\end{equation}
This example shows that an attempt to introduce the so-called (strongly) space normal external forces as ones which can be approximated by the functions from $L^p_b(\R,L^\infty(\Omega))$ (analogously to the (time) strongly normal external forces introduced above) fails since such functions do not give the desired asymptotic compactness, as explained  in Example \ref{Ex5.par}. In particular, this example shows that the class of the so-called {\it spatially absolutely continuous} external forces introduced in \cite{XZS} is {\it insufficient} to get the asymptotic compactness even in the case of simplest 1D linear heat equation.
\end{remark}
Next example will be related with damped wave equations.
\begin{example}\label{Ex5.hyp} Let us consider the following linera damped wave equation on a segment $\Omega=(-\pi,\pi)$ endowed by Dirichlet boundary conditions:
\begin{equation}\label{5.hyp}
\Dt^2 u+\Dt u-\partial_x^2 u=g(t), \ \ u(0)=u'(0)=0,
\end{equation}
where the right-hand side $g$ possesses the explicit description:
\begin{equation}\label{5.g-hyp}
g(t)=
\begin{cases}
\cos(nt)\sin(nx),\ \ t\in[3n\pi,3(n+1)\pi)],\ n\in\Bbb N;\\
0,\ \ \text{otherwise.}
\end{cases}
\end{equation}
Clearly, $g\in L^2_b(\R,L^2(\Omega))$ and the linear equation \eqref{5.hyp} possesses a weak uniform attractor in the energy space $E:=H^1_0(\Omega)\times L^2(\Omega)$. We claim that the trajectory $u(t)$ defined by \eqref{5.hyp} is not precompact in $E$ and, therefore, the strong uniform attractor does not exist. Indeed, if we split the solution $u(t)$ into the Fourier basis:
\begin{equation}\label{5.f}
u(t,x)=\sum_{n=1}^\infty u_n(t)\sin(nx),
\end{equation}
then the explicit computations gives
\begin{equation}\label{5.gn-hyp}
u_n(t)=-\frac{2e^{-(t-3n\pi)/2}\sin\(\frac12\sqrt{4n^2-1}(t-3n\pi)\)}{\sqrt{4n^2-1}}+\frac{\sin(n(t-3n\pi))}n,
\end{equation}
for $t\in[3n,3(n+1))$. Let us fix $t_n:=\pi(3n+2+\frac1{2n})$. Then,
\begin{equation}\label{5.hbad}
u_n(t_n)\ge \frac1n-\frac{2e^{-\pi}}{\sqrt{4n^2-1}}\ge \frac1n\(1-\frac{2e^{-\pi}}{\sqrt{3}}\).
\end{equation}
Estimate \eqref{5.hbad} shows that the sequence $\{u(t_n)\}_{n=1}^\infty$ is not precompact in $H^1_0(\Omega)$. Thus, the asymptotic compactness fails and the strong attractor does not exist.
\end{example}
\begin{remark}\label{Rem5.hyp} The external force $g$ defined by \eqref{5.g-hyp} possesses the additional regularity, namely
\begin{equation}\label{5.g-hreg}
g\in L^\infty(\R,L^\infty(\Omega)).
\end{equation}
In particular, this external force is normal and even strongly normal, so in contrast to the case of parabolic equations, the normality of the exernal forces is not sufficient to obtain the asymptotic compactness in the class of damped wave equations.
\end{remark}
Our next example is related with unbounded domains.
\begin{example}\label{Ex5.unb} Let us consider the following linear heat equation on the whole line $\Omega=\R$:
\begin{equation}\label{5.unbounded}
\Dt u+\alpha u-\partial_x^2 u=g(t),\ \ u(0)=u_0,\ \ \alpha>0.
\end{equation}
To define $g$, we introduce a bump function $V\in C_0^\infty(\R)$ such that $V(0)=0$, $V$ is not zero identically and set $u(t,x)=V(x-t)$,
$u_0(x):=V(x)$, $g(t,x)= g_0(x-t)$, $g_0(x):=-V'(x)+\alpha V(x)-V''(x)$. Then, equation \eqref{5.unbounded} is satisfied with this choice of $u$ and $g$. Moreover,
\begin{equation}\label{5.g-good}
g\in C^\infty_b(\R,C^\infty_b(\R))
\end{equation}
and, in particular it is normal and time regular. However, for a travelling wave solution $u(t,x)=V(x-t)$, the orbit $u(t)$ is not precompact in any Sobolev space $H^s(\R)$, $s\in\R$. Thus, the asymptotic compactness fails and strong attractor does not exist. This example shows that only the smoothness of the external forces $g$ is not sufficient to get the asymptotic compactness and some kind of uniform tail estimates are requred from $g$, see also Remark \ref{Rem4.norm}. This gives another reason why the result of \cite{XZS} is wrong and spatially absolutely continous external forces are {\it insufficient} for the asymptotic compactness in the case of unbounded domains.
\end{example}

We conclude our exposition by one more example which shows that the translation boundedness is {\it not necessary} for the existence of a strong
attractor.
\begin{example}\label{Ex5.unbounded} Let us consider again the linear heat equation on the segment $\Omega=(-\pi,\pi)$ endowed by Dirichlet boundary conditions:
\begin{equation}\label{5.heatagain}
\Dt u-\partial_x u=g(t),\  u\big|_{t=\tau}=u_\tau
\end{equation}
and take the external force in the form
\begin{equation}
g(t)=\sin(x)\begin{cases} 0,\ \ t\le0,\\ t\sin(e^t),\ \  t\ge0.
\end{cases}
\end{equation}
Then, due to averaging of rapid oscillations in time, all trajectories of \eqref{5.heatagain} are uniformly bounded in any $H^s(\Omega)$, $s\in\R$,
so the strong attractor exists (in any phase space $H^s(\Omega)$). However, $g(t)$ is unbounded in time and cannot be translation bounded.
\end{example}

\end{document}